\documentclass[a4paper, 12pt, twoside, notitlepage]{amsart}
\usepackage{latexsym}

\usepackage{amsthm}

\usepackage{amsfonts}
\usepackage{fullpage}
\usepackage{comment}

\usepackage{amsmath,amscd}
\usepackage{amssymb}

\usepackage{graphicx, xcolor}

\usepackage{mathrsfs}
\usepackage[hidelinks]{hyperref}

\hypersetup{pdftitle={Electromagnetic scattering by Minkowski-type bianisotropic media},
 pdfauthor={Manas Kar and Mourad Sini}}

\usepackage{bm}
\usepackage{bbm}
\usepackage{url}

\usepackage[utf8]{inputenc}
\usepackage{mathtools,amssymb}
\usepackage{esint}
\usepackage{tikz}
\usepackage{dsfont}
\usepackage{relsize}
\usepackage{url}
\usepackage{xcolor}
\usepackage{graphicx}
\usepackage{mathrsfs}
\usepackage[shortlabels]{enumitem}
\usepackage{lineno}
\usepackage{amsmath}
\usepackage{enumitem}
\usepackage{amsthm} 
\usepackage{verbatim}
\usepackage{dsfont}
\numberwithin{equation}{section}

\allowdisplaybreaks

\mathtoolsset{showonlyrefs}

\graphicspath{{images/}}

\title[]{Electromagnetic scattering by Minkowski-type bianisotropic media}

\author{Manas Kar}
\address{Indian Institute of Science Education and Research Bhopal, Bhopal By-pass Road, Bhauri Bhopal, 462066,
Madhya Pradesh, India.}
\email{manas@iiserb.ac.in}
\author{Mourad Sini}
\address{Radon Institute, Austrian Academy of Sciences, Altenbergerstrasse 69, Linz, Austria}
\email{mourad.sini@oeaw.ac.at}

\newcommand{\mT}{\mathsf{T}}
\newcommand{\mB}{\mathsf{B}}
\newcommand{\mC}{\mathsf{C}}
\newcommand{\mD}{\mathsf{D}}
\newcommand{\Pperp}{\mathsf{P}_{\perp}}
\newcommand{\Ppar}{\mathsf{P}_{\parallel}}
\renewcommand{\div}{\ensuremath{\mathrm{div}}}
\newcommand{\curl}{\ensuremath{\mathrm{curl}}}

\theoremstyle{plain}
\newtheorem{theorem}{Theorem}

\newtheorem{lemma}[theorem]{Lemma}

\newtheorem{proposition}[theorem]{Proposition}

\theoremstyle{definition}

\theoremstyle{remark}
\newtheorem{remark}[theorem]{Remark}

\newtheorem{assumption}{Assumption}
\DeclareMathOperator{\dist}{dist}

\DeclareMathOperator{\real}{Re}

\def\vint_#1{\mathchoice%
          {\mathop{\kern 0.2em\vrule width 0.6em height 0.69678ex
depth -0.58065ex
                  \kern -0.8em \intop}\nolimits_{\kern -0.4em#1}}%
          {\mathop{\kern 0.1em\vrule width 0.5em height 0.69678ex
depth -0.60387ex
                  \kern -0.6em \intop}\nolimits_{#1}}%
          {\mathop{\kern 0.1em\vrule width 0.5em height 0.69678ex
              depth -0.60387ex
                  \kern -0.6em \intop}\nolimits_{#1}}%
          {\mathop{\kern 0.1em\vrule width 0.5em height 0.69678ex
depth -0.60387ex
                  \kern -0.6em \intop}\nolimits_{#1}}}
\def\vintslides_#1{\mathchoice%
          {\mathop{\kern 0.1em\vrule width 0.5em height 0.697ex depth -0.581ex
                  \kern -0.6em \intop}\nolimits_{\kern -0.4em#1}}%
          {\mathop{\kern 0.1em\vrule width 0.3em height 0.697ex depth -0.604ex
                  \kern -0.4em \intop}\nolimits_{#1}}%
          {\mathop{\kern 0.1em\vrule width 0.3em height 0.697ex depth -0.604ex
                  \kern -0.4em \intop}\nolimits_{#1}}%
          {\mathop{\kern 0.1em\vrule width 0.3em height 0.697ex depth -0.604ex
                  \kern -0.4em \intop}\nolimits_{#1}}}

\newcommand{\aveint}[2]{\mathchoice%
          {\mathop{\kern 0.2em\vrule width 0.6em height 0.69678ex
depth -0.58065ex
                  \kern -0.8em \intop}\nolimits_{\kern -0.45em#1}^{#2}}%
          {\mathop{\kern 0.1em\vrule width 0.5em height 0.69678ex
depth -0.60387ex
                  \kern -0.6em \intop}\nolimits_{#1}^{#2}}%
          {\mathop{\kern 0.1em\vrule width 0.5em height 0.69678ex
depth -0.60387ex
                  \kern -0.6em \intop}\nolimits_{#1}^{#2}}%
          {\mathop{\kern 0.1em\vrule width 0.5em height 0.69678ex
depth -0.60387ex
                  \kern -0.6em \intop}\nolimits_{#1}^{#2}}}

\numberwithin{equation}{section}

\makeatletter
\let\OriginalSetAddresses\@setaddresses
\def\@setaddresses{{\setlength{\bigskipamount}{3pt plus 1pt}\OriginalSetAddresses}}
\makeatother

\begin{document}
\begin{abstract}
We consider the time-harmonic scattering of electromagnetic waves by a bounded stationary inclusion whose effective constitutive law is of Minkowski type. The constant vector $V$ is regarded as a constitutive moving parameter. Such a law arises as an effective description of synthetic moving media and of suitable traveling-wave space--time modulations. Starting from the implicit Minkowski relations, we derive a Maxwell transmission system and its Lippmann--Schwinger volume-integral formulation. The dimensionless parameter $\beta=|V|/c_\Omega$ governs the pointwise constitutive inversion, namely the matrix is invertible precisely for $\beta\neq1$, while on the subcritical branch $0\leq\beta<1$ its inverse admits an explicit convergent expansion. Using the algebraic structure of the cross-product operator, we obtain a spectral characterization of solvability and a quantitative sufficient condition separating the static material contrast from the $V$-dependent contribution which allows to justify the convergence of the corresponding Born series. As an application, we prove that full-aperture electric far-field data at a fixed frequency uniquely determine the support of the inclusion within this perturbative regime without assuming that $V$ is known.
\end{abstract}
\maketitle

	\medskip
		
		{\emergencystretch=1em\noindent\textbf{Keywords.} Minkowski-type media, bianisotropic media, space--time modulation, Maxwell equations, Born series, inverse scattering.\par} 
		

%\tableofcontents
\section{Introduction}
\subsection{The Minkowski-type constitutive model and its physical origin}
To state the mathematical model, we recall that the electric field $E$, the magnetic induction $B$, the magnetic field $H$, and the electric displacement $D$ satisfy Maxwell's equations
\begin{align}\label{main_eqn*}
\left\{
\begin{array}{ll}
	\curl\; E = -B_t, & \div\; B = 0 \quad \text{ in }  \mathbb{R}^3,\\
		\curl\; H = D_t, & \div\; D = 0  \quad \text{ in }  \mathbb{R}^3.
\end{array}
\right.
\end{align}
 Let $\Omega\subset\mathbb R^3$ be a bounded smooth stationary inclusion. Inside $\Omega$ we assume that the effective constitutive response belongs to the Minkowski class, namely
\begin{align}\label{relation*}
\left\{
\begin{array}{ll}
	D + \frac{1}{c^2}V\times H = \varepsilon (E + V \times B), & \text{ in } \Omega\\
\\
	B - \frac{1}{c^2}V\times E = \mu (H - V \times D) & \text{ in } \Omega,
\end{array}
\right.
\end{align}
where $\varepsilon,\mu>0$ and $V\in\mathbb R^3$ is constant. In the present work, $V$ is a constitutive vector and is not assumed to be the mechanical velocity of the boundary $\partial\Omega$. If the medium is  moving, \eqref{relation*} is the classical Minkowski law and $V$ has the usual velocity interpretation; see \cite{Lurie,Lakhtakia-1}. The same moving-type magnetoelectric structure also appears without mechanical translation. Indeed, traveling-wave modulations of both permittivity and permeability lead, after passage to a co-moving frame and homogenization, to effective bianisotropic media \cite{Huidobro2019,Huidobro2021}; special isorefractive space--time crystals may reproduce the response of moving systems exactly \cite{Prudencio2023}. A stationary passive metamaterial exhibiting a strong synthetic moving response has recently been demonstrated experimentally \cite{Yang2026}. A formal derivation of \eqref{relation*} from a traveling profile is given in Appendix~\ref{app:spacetime}. Thus $\Omega$ is regarded below as the spatial support of the effective Minkowski-type response.

In the exterior domain of $\Omega$, the vectors $B$ and $D$ satisfy
\begin{align}\label{relation}
\left\{
\begin{array}{ll}
D= \varepsilon_{0} E , & \text{ in } \mathbb{R}^3\setminus\Omega \\
\\
	B  = \mu_{0} H  & \text{ in } \mathbb{R}^3\setminus\Omega.
\end{array}
\right.
\end{align}
Across $\partial\Omega$, in the absence of surface charges and currents, the normal components of $D$ and $B$ and the tangential components of $E$ and $H$ satisfy the standard Maxwell transmission conditions. We take $\varepsilon$ and $\mu$ to be positive scalar constants. The time-domain formulation below is the nondispersive idealization of the effective constitutive law. For a synthetic realization the effective parameters may in general depend on frequency; since all scattering and inverse results in this article are formulated at one fixed frequency, $\varepsilon$, $\mu$ and $V$ may equivalently be read as the effective values at the operating frequency, once the constitutive reduction has been performed.
For the fixed-domain transmission formulation used below, write
\[
 \varepsilon_{\Omega}(x):=\varepsilon\chi_{\Omega}(x)
 +\varepsilon_0(1-\chi_{\Omega}(x)),\qquad
 \mu_{\Omega}(x):=\mu\chi_{\Omega}(x)
 +\mu_0(1-\chi_{\Omega}(x)),
 \qquad c=(\varepsilon_0\mu_0)^{-1/2}.
\]
Here $V$ enters the constitutive law, while the support $\Omega$ is fixed in the time-harmonic problem. Differentiating \eqref{relation*} and the exterior constitutive relations with respect to time gives
\begin{align}\label{relation1}
\left\{
\begin{array}{ll}
	D_t + \frac{1}{c^2}V\chi_{\Omega}\times H_t = \varepsilon_{\Omega} (E_t + V\chi_{\Omega} \times B_t), & \text{ in } \mathbb{R}^3 \\
\\
	B_t - \frac{1}{c^2}V\chi_{\Omega}\times E_t = \mu_{\Omega} (H_t - V \chi_{\Omega} \times D_t) & \text{ in } \mathbb{R}^3,
\end{array}
\right.
\end{align}
where $\chi_{\Omega}$ is the characteristic function of $\Omega$. Note that if $V = (v_1,v_2,v_3)\in\mathbb{R}^3$ then the vector product $V\times$ represents a matrix
$$
V\times=\left(\begin{array}{ccc}
0 & -v_3 & v_2 \\
v_3 & 0 & -v_1 \\
-v_2 & v_1 & 0
\end{array}\right)
$$
for $x\in \Omega$ and $V\chi_{\Omega}\times=0_{3\times 3}$ for $x\in\mathbb{R}^3\setminus \overline{\Omega}.$
We write the above relations in the form
\begin{equation}
 \begin{pmatrix}
1 & \mu_{\Omega} V\chi_{\Omega}\times \\ \\
-\varepsilon_{\Omega} V\chi_{\Omega}\times & 1 
\end{pmatrix}
\begin{pmatrix}
B_t \\ \\
D_t 
\end{pmatrix}
=
\begin{pmatrix}
\frac{1}{c^2}V\chi_{\Omega}\times & \mu_{\Omega} \\ \\
\varepsilon_{\Omega} & -\frac{1}{c^2}V\chi_{\Omega}\times 
\end{pmatrix}
\begin{pmatrix}
E_t\\ \\
H_t
\end{pmatrix},
\end{equation}
where $$\begin{pmatrix}
1 & \mu_{\Omega} V\chi_{\Omega}\times \\ \\
-\varepsilon_{\Omega} V\chi_{\Omega}\times & 1 
\end{pmatrix}
$$
means the $6\times 6$ matrix
\[
\begin{pmatrix}
I_{3\times 3} & \mu_{\Omega} V\chi_{\Omega}\times \\ \\
-\varepsilon_{\Omega} V\chi_{\Omega}\times & I_{3\times 3}
\end{pmatrix}.
\]
Similarly, the matrix
\[
\begin{pmatrix}
\frac{1}{c^2}V\chi_{\Omega}\times & \mu_{\Omega} \\ \\
\varepsilon_{\Omega} & -\frac{1}{c^2}V\chi_{\Omega}\times 
\end{pmatrix}
\]
is understood as a $6\times 6$ matrix of the form
\[
\begin{pmatrix}
\frac{1}{c^2}V\chi_{\Omega}\times & \mu_{\Omega} I_{3\times 3} \\ \\
\varepsilon_{\Omega} I_{3\times 3} & -\frac{1}{c^2}V\chi_{\Omega}\times 
\end{pmatrix}.
\]
Replacing $B_t$ and $D_t$ in the system \eqref{main_eqn*}, we get
\begin{equation}\label{eq_matrix}
 \begin{pmatrix}
1 & \mu_{\Omega} V\chi_{\Omega}\times \\ \\
-\varepsilon_{\Omega} V\chi_{\Omega}\times & 1 
\end{pmatrix}
\begin{pmatrix}
-\curl\; E  \\ \\
\curl\; H
\end{pmatrix}
=
\begin{pmatrix}
\frac{1}{c^2}V\chi_{\Omega}\times & \mu_{\Omega} \\ \\
\varepsilon_{\Omega} & -\frac{1}{c^2}V\chi_{\Omega}\times 
\end{pmatrix}
\begin{pmatrix}
E_t\\ \\
H_t
\end{pmatrix}.
\end{equation}
This is our starting Maxwell system for a stationary inclusion carrying a Minkowski-type constitutive response. 
\subsection{The time-harmonic regime}
We are interested in time-harmonic wave propagation in the presence of the inclusion $\Omega$. We seek solutions of the form $$E:=\real(E(\cdot, \omega)e^{-i\omega t})$$ and $$H:=\real(H(\cdot, \omega)e^{-i\omega t}).$$
%Define the Fourier transform in time:
%\[
%E(x,\omega) = (\mathcal{F}_t\mathcal{E})(x;\omega) = \int_{\mathbb{R}}\mathcal{E}(x,t) e^{i\omega t} dt,
%\]
 %and
 %\[
%H(x,\omega) = (\mathcal{F}_t\mathcal{H})(x;\omega) = \int_{\mathbb{R}}\mathcal{H}(x,t) e^{i\omega t} dt.
%\]
%Use 
%\[
%\mathcal{F}_t(\frac{\partial u}{\partial t}) = -i\omega\mathcal{F}_t u
%\]
Therefore $E(\cdot,\omega)$ and $H(\cdot,\omega)$ satisfy
{\setlength{\arraycolsep}{4pt}
\begin{equation}\label{eq_matrix2}
 \begin{pmatrix}
1 & \mu_{\Omega} V\chi_{\Omega}\times \\ \\
-\varepsilon_{\Omega} V\chi_{\Omega}\times & 1 
\end{pmatrix}
\begin{pmatrix}
-\curl\; E(x,\omega) \\ \\
\curl\; H(x,\omega)
\end{pmatrix}
=
\begin{pmatrix}
\frac{1}{c^2}V\chi_{\Omega}\times & \mu_{\Omega} \\ \\
\varepsilon_{\Omega} & -\frac{1}{c^2}V\chi_{\Omega}\times 
\end{pmatrix}
\begin{pmatrix}
-i\omega E(x,\omega)\\ \\
-i\omega H(x,\omega)
\end{pmatrix}.
\end{equation}
}

If the matrix 
\begin{equation}\label{maTRiX}
 \begin{pmatrix}
1 & \mu_{\Omega} V\chi_{\Omega}\times \\ \\
-\varepsilon_{\Omega} V\chi_{\Omega}\times & 1 
\end{pmatrix}
\end{equation}
is invertible, then the above equation formally writes as follows
\begin{equation}\label{eq_matrix3}
\begin{pmatrix}
\curl\; E(x,\omega) \\ \\
\curl\; H(x,\omega)
\end{pmatrix}
=
 \begin{pmatrix}
-1 & \mu_{\Omega} V\chi_{\Omega}\times \\ \\
\varepsilon_{\Omega} V\chi_{\Omega}\times & 1 
\end{pmatrix}^{-1}
\begin{pmatrix}
\frac{1}{c^2}V\chi_{\Omega}\times & \mu_{\Omega} \\ \\
\varepsilon_{\Omega} & -\frac{1}{c^2}V\chi_{\Omega}\times 
\end{pmatrix}
\begin{pmatrix}
-i\omega E(x,\omega)\\ \\
-i\omega H(x,\omega)
\end{pmatrix}.
\end{equation}
Simplifying further we get, 
%\footnote{So far we have not used the equations $\div\; B = 0$ and $\div\; D = 0 $ to derived the Time harmonic Maxwell system under moving dielectric.}
\begin{equation}\label{eq_matrix4}
\begin{pmatrix}
\curl\; E(x,\omega) \\ \\
\curl\; H(x,\omega)
\end{pmatrix}
=
-i\omega
\tilde{A}\begin{pmatrix}
 E(x,\omega)\\ \\
 H(x,\omega)
\end{pmatrix},
\end{equation}
where
\begin{equation}\label{Coefficient_matrix}
\tilde{A} := \begin{pmatrix}
-1 & \mu_{\Omega} V\chi_{\Omega}\times \\ \\
\varepsilon_{\Omega} V\chi_{\Omega}\times & 1 
\end{pmatrix}^{-1}
\begin{pmatrix}
\frac{1}{c^2}V\chi_{\Omega}\times & \mu_{\Omega} \\ \\
\varepsilon_{\Omega} & -\frac{1}{c^2}V\chi_{\Omega}\times 
\end{pmatrix}.
\end{equation}
The algebraic condition behind this reduction can be stated exactly.  In the following lemma the characteristic function is suppressed, since the matrix is the identity outside $\Omega$.
\begin{lemma}\label{lem:constitutive-inverse}
Let $K:=V\times$, let
\[
 c_{\Omega}:=\frac{1}{\sqrt{\varepsilon\mu}},
 \qquad
 \beta:=\frac{|V|}{c_{\Omega}}=|V|\sqrt{\varepsilon\mu},
\]
and define
\[
 \mathbb{A}_V:=
 \begin{pmatrix}
 I_{3\times3}&\mu K\\
 -\varepsilon K&I_{3\times3}
 \end{pmatrix}.
\]
Then $\mathbb{A}_V$ is invertible if and only if $\beta\neq1$.  More precisely, with
\[
 \mathbb{S}_V:=
 \begin{pmatrix}
 0&-\mu K\\
 \varepsilon K&0
 \end{pmatrix},
 \qquad \mathbb{A}_V=I-\mathbb{S}_V,
\]
one has, whenever $\beta\neq1$,
\begin{equation}\label{exact-constitutive-inverse}
 \mathbb{A}_V^{-1}
 =I+\frac{\mathbb{S}_V+\mathbb{S}_V^2}{1-\beta^2}.
\end{equation}
In particular, in the subcritical regime $0\leq\beta<1$,
\begin{equation}\label{constitutive-neumann-series}
 \mathbb{A}_V^{-1}=\sum_{j=0}^{\infty}\mathbb{S}_V^j,
\end{equation}
and the series converges in the natural weighted Euclidean norm.
\end{lemma}
\begin{proof}
The cross-product matrix satisfies
\[
 K^2=V\otimes V-|V|^2I_{3\times3},
 \qquad K^3=-|V|^2K.
\]
Consequently, $\mathbb S_V^3=\beta^2\mathbb S_V$. Let
\[
 \mathbb D:=\operatorname{diag}
 \big(\sqrt\varepsilon I_{3\times3},\sqrt\mu I_{3\times3}\big).
\]
A direct calculation gives
\[
 \mathbb D\mathbb S_V\mathbb D^{-1}
 =\sqrt{\varepsilon\mu}
 \begin{pmatrix}0&-K\\ K&0\end{pmatrix}.
\]
This matrix is real symmetric and has spectrum $\{0,\beta,-\beta\}$. Equivalently, with $\Ppar$ and $\Pperp$ denoting the projections onto $\operatorname{span}\{V\}$ and $V^\perp$, respectively,
\[
 I_{3\times3}+\varepsilon\mu K^2
 =\Ppar+(1-\beta^2)\Pperp.
\]
Hence $I-\mathbb S_V$ is invertible exactly when $\beta\neq1$. Formula \eqref{exact-constitutive-inverse} follows from $\mathbb S_V^3=\beta^2\mathbb S_V$, and \eqref{constitutive-neumann-series} follows when $\beta<1$. In the weighted Euclidean norm induced by $\mathbb D$,
\begin{equation}\label{constitutive-inverse-norm}
 \|\mathbb A_V^{-1}\|_{*,2}=\frac{1}{1-\beta},
 \qquad 0\leq\beta<1.
\end{equation}
\end{proof}
Thus the subcritical condition $|V|<c_\Omega$ is already sufficient for the pointwise inversion of the Minkowski constitutive system. No additional entrywise smallness condition on the matrix \eqref{maTRiX} is needed. The dimensionless number $\beta=|V|/c_\Omega$ measures the strength of the moving-type coupling relative to the interior wave speed. The value $\beta=1$ is a constitutive threshold: after the symmetrization used in the proof, the spectrum is $\{0,\beta,-\beta\}$ and \eqref{constitutive-inverse-norm} degenerates like $(1-\beta)^{-1}$ as $\beta\uparrow1$. In the space--time interpretation of Appendix~\ref{app:spacetime}, the same denominator marks the approach to the local luminal regime. The near-critical problem is not treated here; the stronger condition introduced below concerns only the inversion of the Lippmann--Schwinger volume operator and is perturbative before this threshold.

One of our main contributions in this article is to understand the transmission problem corresponding to equations \eqref{eq_matrix4} and \eqref{Coefficient_matrix}. 
From \eqref{eq_matrix4} and \eqref{Coefficient_matrix}, we see that $E$ and $H$ can be in the energy space $W^{1,p}_{loc}(\curl,\mathbb{R}^3)$ only if $E\wedge\nu$ and $H\wedge\nu$ are continuous across $\partial\Omega$.  We distinguish the background wave speed $c$, the interior wave speed $c_{\Omega}=1/\sqrt{\varepsilon\mu}$, and the constitutive drift scale $c_V:=|V|$. We work in the subcritical regime
\[
 |V|<c_{\Omega}<c.
\]
For a moving medium the first inequality is the usual subluminal condition; for a synthetic medium it is the corresponding algebraic branch of the effective Minkowski law. In either interpretation it selects the subcritical branch of the exact pointwise criterion $|V|\neq c_\Omega$ in Lemma~\ref{lem:constitutive-inverse}; the additional quantitative restriction used later is imposed only to obtain a convergent Born series for the volume-integral equation.  We therefore seek solutions of \eqref{eq_matrix4} and \eqref{Coefficient_matrix} of the form $E:=E^s+E^i$ and $H:=H^s+H^i$, where $(E^i,H^i)$ is an incident electromagnetic field solving the Maxwell system in the absence of $\Omega$, and $(E^s,H^s)$ satisfies the usual Silver--M\"uller radiation condition
\begin{equation}\label{Siver_Muller}
    \left\{\begin{array}{l}\frac{x}{|x|} \times H^s+\sqrt{\frac{\varepsilon_0}{\mu_0}} E^s=O\left(\frac{1}{|x|^2}\right) \text { as }|x| \rightarrow \infty, \\ \frac{x}{|x|} \times E^s -\sqrt{\frac{\mu_0}{\varepsilon_0}} H^s=O\left(\frac{1}{|x|^2}\right) \text { as }|x| \rightarrow \infty .\end{array}\right.
\end{equation}
\subsection{Transmission problems}
Let $\Omega$ be a bounded $C^1$-smooth domain in $\mathbb{R}^3$, and $\nu$ be the outward unit normal to the boundary $\partial\Omega$. Assume that $\partial \Omega$ is connected. In this article, we study the transmission boundary value problem for the Maxwell system. We seek the interior fields $(E,H)$ and the scattered exterior fields $(E^s,H^s)$ for the Minkowski-type inclusion. The vector field $(E,H)$ satisfies \eqref{eq_matrix4} inside $\Omega$, whereas outside $\Omega$ the constitutive vector is zero and $(E^s,H^s)$ satisfies the classical Maxwell system together with the Silver--M\"uller radiation condition. The transmission boundary value problem reads as
\begin{equation}\label{model_EQ}
\left\{\begin{array}{l}
\curl\; \begin{pmatrix}
 E \\ \\
 H
\end{pmatrix}
=
-i\omega
\tilde{A}^{+}\begin{pmatrix}
E\\ \\
 H
\end{pmatrix}, \text { in } \Omega \\ \\
\curl\; \begin{pmatrix}
 E^{s} \\ \\
 H^{s}
\end{pmatrix}
=
-i\omega
\tilde{A}^{-}\begin{pmatrix}
E^{s}\\ \\
 H^{s}
\end{pmatrix}, \text { in } \mathbb{R}^3 \backslash \overline{\Omega}\\ \\
\nu \times\left. E^{s}\right|_{\partial \Omega}-\nu \times\left. E\right|_{\partial \Omega}=-\nu \times E^{i} \\
\nu \times\left. H^{s}\right|_{\partial \Omega}-\nu \times\left. H\right|_{\partial \Omega}=-\nu \times H^{i} \\
E^{s}, H^{s} \ \text{ satisfy the Silver-M\"uller radiation conditions \eqref{Siver_Muller} at infinity, }
\end{array}\right. 
\end{equation}
where 
\[
\tilde{A}^{+} := \begin{pmatrix}
-1 & \mu V\chi_{\Omega}\times \\ \\
\varepsilon V\chi_{\Omega}\times & 1 
\end{pmatrix}^{-1}
\begin{pmatrix}
\frac{1}{c^2}V\chi_{\Omega}\times & \mu \\ \\
\varepsilon & -\frac{1}{c^2}V\chi_{\Omega}\times 
\end{pmatrix}
\]
and
\[
\tilde{A}^{-} := \begin{pmatrix}
-1 & 0 \\ \\
0 & 1 
\end{pmatrix}^{-1}
\begin{pmatrix}
0 & \mu_0 \\ \\
\varepsilon_0 & 0 
\end{pmatrix}=\begin{pmatrix}
0 & -\mu_0 \\ \\
\varepsilon_0 & 0 
\end{pmatrix}.
\]
%Note that $\varepsilon$ and $\mu$ are the electric permittivity and magnetic permeability of the domain $\Omega$ respectively, $\varepsilon_0$ and $\mu_0$ are the electric permittivity and magnetic permeability in the vacuum respectively. 

For the volume estimates, introduce the relative material parameters and the two dimensionless quantities
\begin{equation}\label{dimensionless-parameters}
 r_{\varepsilon}:=\frac{\varepsilon}{\varepsilon_0},\qquad
 r_{\mu}:=\frac{\mu}{\mu_0},\qquad
 \delta_{\rm st}:=\max\{|r_{\varepsilon}-1|,|r_{\mu}-1|\},\qquad
 r_*:=\max\{r_{\varepsilon},r_{\mu}\},
\end{equation}
\begin{equation}\label{velocity-parameters}
 \beta:=\frac{|V|}{c_{\Omega}},
 \qquad
 \vartheta:=\left|1-\frac{c_{\Omega}^2}{c^2}\right|.
\end{equation}
Set
\begin{equation}\label{weighted-coefficients}
 a:=\mu_0r_{\varepsilon},\qquad b:=\varepsilon_0r_{\mu},
\end{equation}
and equip $L^p(\Omega)^6$ with the equivalent weighted norm
\begin{equation}\label{weighted-norm}
 \left\|\binom{F_1}{F_2}\right\|_{p,*}
 :=\sqrt b\,\|F_1\|_{L^p(\Omega)}+\sqrt a\,\|F_2\|_{L^p(\Omega)}.
\end{equation}
Let
\begin{equation}\label{free-volume-operator}
 (\mathcal V_{\omega}F)(x):=\omega\int_{\Omega}G(x-y)F(y)\,dy
\end{equation}
be the free Maxwell volume potential. The integral is understood through distributional derivatives of the Helmholtz Newton potential, equivalently by \eqref{volume-potential-components}; it is not an ordinary absolutely convergent integral at points of $\Omega$. As the Newtonian operator maps $L^p(\Omega)$ continuously to $W^{2, p}(\Omega)$, via the Calderon-Zygmund inequality, then $\mathcal V_{\omega}$ maps $L^p(\Omega)$ to itself. Denote its norm on the weighted space by
\begin{equation}\label{volume-operator-constant}
 \mathfrak C_{p,\omega,\Omega}^{*}
 :=\|\mathcal V_{\omega}\|_{\mathcal L(L^p_*(\Omega)^6)}.
\end{equation}
A concrete bound for $\mathfrak C_{p,\omega,\Omega}^{*}$ is given in \eqref{explicit-volume-bound} below.

Define the static contrast matrix
\[
 \mathcal M_{\rm st}:=
 \begin{pmatrix}
 (r_{\varepsilon}-1)I_{3\times3}&0\\
 0&(r_{\mu}-1)I_{3\times3}
 \end{pmatrix},
\]
the corresponding static volume operator $G_{\rm st}:=\mathcal V_{\omega}\mathcal M_{\rm st}$, and
\begin{equation}\label{static-resolvent}
 K_{\rm st}:=I-G_{\rm st}.
\end{equation}

\begin{assumption}\label{assum}
The following conditions are imposed.
\begin{enumerate}[(i)]
 \item The Minkowski parameter lies on the subcritical branch
 \begin{equation}\label{regimes-1}
  0\leq\beta=\frac{|V|}{c_{\Omega}}<1.
 \end{equation}
 \item The static Lippmann--Schwinger operator $K_{\rm st}$ is invertible on $L^p_*(\Omega)^6$.  We set
 \begin{equation}\label{kappa-static}
  \kappa_{\rm st}:=\|K_{\rm st}^{-1}\|_{\mathcal L(L^p_*(\Omega)^6)}.
 \end{equation}
 \item The $V$-dependent perturbation satisfies
 \begin{equation}\label{regimes-2}
  \eta_V:=\kappa_{\rm st}\,\mathfrak C_{p,\omega,\Omega}^{*}
  \vartheta r_*\frac{\beta}{1-\beta}<1.
 \end{equation}
\end{enumerate}
\end{assumption}
Condition \eqref{regimes-1} is exactly the subcritical branch of the algebraic invertibility criterion in Lemma~\ref{lem:constitutive-inverse}. Assumption~\ref{assum}(ii) concerns only the associated static configuration $V=0$ and is the natural non-resonance condition for its volume-integral equation. Assumption~\ref{assum}(iii) then treats the moving-type, $V$-dependent part as a perturbation of that full static problem; it does not require the static contrast $\delta_{\rm st}$ itself to be small.

For a completely explicit, although stronger, sufficient condition one may use
\begin{equation}\label{explicit-total-smallness}
 \mathfrak C_{p,\omega,\Omega}^{*}
 \left(\delta_{\rm st}+\vartheta r_*\frac{\beta}{1-\beta}\right)<1.
\end{equation}
Indeed, $\mathfrak C_{p,\omega,\Omega}^{*}\delta_{\rm st}<1$ then gives
$\kappa_{\rm st}\leq(1-\mathfrak C_{p,\omega,\Omega}^{*}\delta_{\rm st})^{-1}$, and \eqref{regimes-2} follows.  Formula \eqref{explicit-total-smallness} displays separately the usual static contrast and the moving-type contrast
\[
 \delta_{\rm mov}:=\vartheta r_*\frac{\beta}{1-\beta}.
\]
The latter vanishes when either $V=0$ or $c_{\Omega}=c$.
\bigskip

Our main result for the direct problem is as follows.
\begin{theorem}\label{Main_thm_direct}
Let $\Omega$ be an arbitrary bounded $C^1$-smooth domain in $\mathbb{R}^3$, let $1<p<\infty$, and assume Assumption~\ref{assum}(i)--(ii). Set $\mathcal A_V:=K_{\rm st}^{-1}G_{\rm mov}$, see (\ref{static-resolvent}) and (\ref{GM-decomposition}). If
\begin{equation}\label{intro-exact-spectral-condition}
 1\notin\sigma(\mathcal A_V),
\end{equation}
then problem \eqref{model_EQ} is uniquely solvable in $W^{1,p}_{loc}(\curl,\mathbb{R}^3)$. In addition, there exists a constant $C>0$, depending on $\partial\Omega$, the material parameters, $\omega$, $\kappa_{\rm st}$, and $\|(I-\mathcal A_V)^{-1}\|_*$, such that the solution satisfies
\[
\left\|E\right\|_{L^{p}( \Omega)}+\left\|H\right\|_{L^{p}(\Omega)} \leq C\left[\|\nu \times E^{i}\|_{L_{\tan }^{p, \operatorname{Div}}(\partial \Omega)}+\|\nu \times H^{i}\|_{L_{\tan }^{p, \operatorname{Div}}(\partial \Omega)}\right]
\]
and
\[
\left\|E^{s}\right\|_{L_{loc}^{p}( \mathbb{R}^3 \backslash \overline{\Omega})}+\left\|H^{s}\right\|_{L_{loc}^{p}(\mathbb{R}^3 \backslash \overline{\Omega})} \preceq \left[\|\nu \times E^{i}\|_{L_{\tan }^{p, \operatorname{Div}}(\partial \Omega)}+\|\nu \times H^{i}\|_{L_{\tan }^{p, \operatorname{Div}}(\partial \Omega)}\right].
\]
\end{theorem}
If $r_{\rm sp}(\mathcal A_V)<1$, the solution is represented by the motion Born series \eqref{motion-Born-series}. In particular, the operator-norm condition \eqref{exact-motion-Neumann-condition}, and hence Assumption~\ref{assum}(iii), implies both \eqref{intro-exact-spectral-condition} and the quantitative estimate \eqref{solution-weighted-bound}.
Here, for $1<p<\infty$, we consider the function space as
$$L_{\tan }^{p, \operatorname{Div}}(\partial \Omega):=\left\{f \in L_{\tan }^p(\partial \Omega) ; \operatorname{Div} f \in L^p(\partial \Omega)\right\},$$ see more details in the Section \ref{sec2}. 

In order to prove this theorem, we will first prove the following two propositions. In the first proposition, we deduce the Lippmann--Schwinger system of equations related to the problem \eqref{model_EQ}.
\begin{proposition}\label{Lipp_s1}
For each $1<p<\infty$, the transmission problem \eqref{model_EQ} is equivalent to the Lippmann--Schwinger equation
\begin{equation}\label{Lipp_s}
   \begin{pmatrix}
 E-E^{i}\\ \\
 H-H^{i}
\end{pmatrix}(x)
=
\omega
\int_{\Omega}
G(x-y)
\mathcal{M} \begin{pmatrix}
 E(y)\\ \\
 H(y)
\end{pmatrix} dy, \ \ \text{for}\ x\in\Omega,
\end{equation}
for the unknown $(E,H)\in L^p(\Omega)^6$. More precisely, the restriction to $\Omega$ of every solution of \eqref{model_EQ} satisfies \eqref{Lipp_s}, and every $L^p(\Omega)^6$ solution of \eqref{Lipp_s} extends through the volume representation to a radiating solution of \eqref{model_EQ}. Here $E^{i},H^{i}$ are the incident fields. The matrix $G$ is the outgoing fundamental matrix for the background Maxwell system, defined in \eqref{fund_soll}, and
\begin{equation}\label{M-introduction-formula}
 \mathcal M=\mathcal M_{\rm st}+\mathcal M_{\rm mov},
\end{equation}
where
\[
 \mathcal M_{\rm st}:=
 \begin{pmatrix}
 (r_{\varepsilon}-1)I_{3\times3}&0\\
 0&(r_{\mu}-1)I_{3\times3}
 \end{pmatrix}
\]
and
\begin{equation}\label{Mmov-introduction-formula}
 \mathcal M_{\rm mov}
 =\vartheta\frac{\mT+\mT^2}{1-\beta^2}
 \begin{pmatrix}
 r_{\varepsilon}I_{3\times3}&0\\
 0&r_{\mu}I_{3\times3}
 \end{pmatrix},
\end{equation}
with
\[
 \mT=
 \begin{pmatrix}
 0&\mu_0r_{\varepsilon}V\times\\
 -\varepsilon_0r_{\mu}V\times&0
 \end{pmatrix}.
\]
In particular, \eqref{Mmov-introduction-formula} is exact for every $\beta<1$ and admits the convergent expansion
\[
 \mathcal M_{\rm mov}
 =\vartheta\sum_{j=0}^{\infty}\mT^{j+1}
 \begin{pmatrix}
 r_{\varepsilon}I_{3\times3}&0\\
 0&r_{\mu}I_{3\times3}
 \end{pmatrix}.
\]
\end{proposition}
\begin{remark}[The limit $c_{\Omega}\to c$]\label{rem:static-limit}
If $c_{\Omega}=c$, then $\vartheta=0$ and \eqref{Mmov-introduction-formula} gives
\[
 \mathcal M_{\rm mov}=0,
 \qquad
 \mathcal M=\mathcal M_{\rm st}.
\]
Thus the $V$-dependent part of the electromagnetic contrast vanishes exactly, although $V$ itself need not be zero. A nontrivial static impedance contrast at $c_{\Omega}=c$ is possible if one allows arbitrary positive material parameters. If one additionally imposes $\varepsilon\geq\varepsilon_0$ and $\mu\geq\mu_0$, $c_{\Omega}=c$ instead forces $\varepsilon=\varepsilon_0$ and $\mu=\mu_0$.  More generally, if $c_{\Omega,n}\to c$ and $\beta_n\leq\beta_*<1$, then
\begin{equation}\label{M-static-limit}
 \|\mathcal M_n-\mathcal M_{{\rm st},n}\|_*
 \leq \vartheta_n r_{*,n}\frac{\beta_n}{1-\beta_n}\longrightarrow0,
\end{equation}
provided $r_{*,n}$ remains bounded.  Consequently, the corresponding volume operators converge in operator norm. If, in addition, the static resolvent norms remain uniformly bounded and the Neumann margins in \eqref{regimes-2} are uniform, the resolvent identity yields convergence of the total fields to those of the associated static dielectric problems in $L^p_*(\Omega)^6$.
\end{remark}
\bigskip

The second proposition deals with the solvability of the Lippmann--Schwinger system of equations.
\begin{proposition}\label{well_posed}
Suppose that Assumption~\ref{assum}(i)--(ii) holds and set $\mathcal A_V:=K_{\rm st}^{-1}G_{\rm mov}$. The Lippmann--Schwinger operator is boundedly invertible if and only if $1\notin\sigma(\mathcal A_V)$. In that case, equation \eqref{Lipp_s} is uniquely solvable for every admissible incident field, and its solution extends to a unique radiating solution in $W^{1,p}_{\rm loc}(\curl,\mathbb{R}^3)$. The conclusion holds, in particular, under the operator-norm condition \eqref{exact-motion-Neumann-condition}, and hence under Assumption~\ref{assum}.
\end{proposition}

The forward scattering problem for the classical Maxwell system with a dielectric at rest is well studied; see, for instance, \cite{MR2059447, MR1822275, MR0700400, MR1438894, Emilio:Irina:Marius:Qiang:2012, +2001+423+446, MR1317629}.  The Newton-potential, layer-potential and Calder\'on--Zygmund estimates used below belong to this classical theory.  Closely related integral formulations for chiral media at rest can be found in \cite{Ammari96time-harmonicelectromagnetic, MR1616499}.  Direct and inverse scattering for more general inhomogeneous bi-anisotropic Maxwell constitutive laws has also been studied; see, for instance, \cite{Nguyen2019Bianisotropic}.  Classical moving-media problems have been studied in \cite{MR0807903, MR0542950, MR0773393}; for uniformly moving chiral media, we refer to \cite{Lakhtakia-1, Lakhtakia-2}. More recently, traveling-wave space--time modulations have provided a second route to moving-type electromagnetic response. Modulating both $\varepsilon$ and $\mu$ produces, in the long-wave regime, an effective bianisotropic medium with synthetic drift \cite{Huidobro2019,Huidobro2021}. The relation with moving systems is particularly transparent for Minkowskian isorefractive space--time crystals \cite{Prudencio2023}; see Appendix~\ref{app:spacetime}. The recent experimental observation of a passive synthetic moving metamaterial further supports the use of a stationary effective model of the form \eqref{relation*} \cite{Yang2026}.
\bigskip

\noindent The algebraic structure of the Minkowski constitutive law allows us to separate the electromagnetic volume contrast into its static and \(V\)-dependent components. An important purpose of this decomposition is to obtain a constructive Born-series representation of the scattering fields, rather than using a perturbative argument solely to establish solvability. By incorporating the full static Maxwell transmission problem into the reference resolvent and expanding only the moving-type contribution, we obtain a convergent series without requiring the static material contrast to be small, provided the corresponding static operator is invertible. This representation provides successive approximations of the electromagnetic fields, with quantitative control of the approximation error, and allows the influence of the constitutive vector on the scattered fields and their far-field patterns to be analyzed systematically. It is important to distinguish the spectral condition ensuring solvability from the stronger conditions guaranteeing convergence of the Born series. In particular, the perturbative bounds imposed for the latter are sufficient conditions for the validity of this representation and should not be interpreted as necessary conditions for well-posedness of the underlying Maxwell transmission problem. \footnote{In a forthcoming work, we establish well-posedness under more general, nonperturbative assumptions, without requiring Born-series convergence, and investigate the associated inverse scattering problem; see the end of Section 1.4.}

%The contribution of the present work is to identify explicitly the volume contrast generated by the Minkowski constitutive law and to quantify its effect within a Lippmann--Schwinger framework.  The identities for $V\times$ yield the exact pointwise criterion $|V|\neq c_{\Omega}$ and its convergent expansion for $|V|<c_{\Omega}$, an explicit decomposition into static and $V$-dependent contrasts, the exact spectral condition \eqref{intro-exact-spectral-condition}, and a Born series whose convergence is controlled by the separated parameter in \eqref{regimes-2}.  Thus the classical potential estimates are used as an analytic tool, while the new part is the algebraic and operator-theoretic treatment of the moving-type coupling.  We restrict here to the time-harmonic regime.

\subsection{Inverse problems}
 The problem \eqref{model_EQ} models the scattering of electromagnetic waves by a penetrable bounded inclusion $\Omega$ carrying the effective Minkowski-type response \eqref{relation*}. Here and subsequently, we will denote by $\omega>0$ the frequency of the electromagnetic wave. The strictly positive numbers $\varepsilon,\mu$ and $\varepsilon_0,\mu_0$ are the interior and exterior permittivities and permeabilities. The reference interior wave number $k$ and the exterior wave number $k_0$ are defined by $k^2=\omega^2 \varepsilon \mu, k_{0}^2=\omega^2 \varepsilon_{0} \mu_{0}$. Let $d \in \mathbb{S}^2$ be the incident direction and $\boldsymbol p \in \mathbb{S}^2$ a fixed polarization seed. The actual electric polarization is its projection onto $d^\perp$. We define the linearly polarized incident plane electric wave as
\begin{equation}\label{E_inci}
  \begin{aligned}
E^{i}(x ; d, \boldsymbol p): & =\frac{i}{k_{0}} \operatorname{curl} \operatorname{curl}\left[\boldsymbol p e^{i k_{0}\langle x, d\rangle}\right] \\
& =i k_{0}(d \times \boldsymbol p) \times d e^{i k_{0}\langle x, d\rangle}, \quad x \in \mathbb{R}^3,
\end{aligned}  
\end{equation}
and the corresponding incident plane magnetic wave
\begin{equation}\label{H_inci}
  \quad H^{i}(x ; d, \boldsymbol p):=\frac{k_{0}}{\omega \mu_{0}} \operatorname{curl}\left[\boldsymbol p e^{i k_{0}\langle x, d\rangle}\right]=\frac{i k_{0}^2}{\omega \mu_{0}} d \times \boldsymbol p e^{i k_{0}\langle x, d\rangle}, \quad x \in \mathbb{R}^3. 
\end{equation}
Such incident waves will produce (radiating) scattered fields $E^{s}, H^{s}$ in the exterior of $\Omega$ and transmitted fields $E, H$ inside $\Omega$. The vector fields $E,H$ in $\Omega$ satisfy the Minkowski-type Maxwell system, while $E^s,H^s$ satisfy the background Maxwell system outside, together with the transmission boundary conditions
\[
\nu \times\left(E^{i}+E^{s}\right)=\nu \times E \ \text{and} \  \nu \times\left(H^{i}+H^{s}\right)= \nu \times H \ \text{on} \ \partial \Omega.
\]
%The boundary conditions in \eqref{model_EQ} are obtained with
%\[
%{f}:=-\nu \times E^{\text {in }} \ \text{and} \  {g}:=-\nu %\times H^{\text {in }} \ 
%\text{on} \ \partial \Omega_{\infty}
%\]
% and zero elsewhere on $\partial \Omega$. Once again, we shall work in the context of scatterers with Lipschitz continuous boundaries. In particular, ${f}, {g}$ above are, generally speaking, discontinuous vector fields.
Our main result concerning the inverse obstacle problem for the transmission problem in electromagnetic scattering reads as follows. First recall that scattered electromagnetic waves $(E^{s}, H^{s})$ satisfying the Silver-M\"uller radiation condition have asymptotic expansions at infinity 
\[
E^{s}(x, d)=\frac{e^{i k_{0}|x|}}{|x|}\left[E_{\infty}\left(\theta, d\right)+O\left(\frac{1}{|x|}\right)\right],\ \text{as} \ |x| \rightarrow \infty
\]
and
\[
H^{s}(x, d)=\frac{e^{i k_{0}|x|}}{|x|}\left[H_{\infty}\left(\theta, d\right)+O\left(\frac{1}{|x|}\right)\right],\ \text{as} \ |x| \rightarrow \infty, 
\]
where $\theta:=\frac{x}{|x|}$ is the observation direction. We denote by $\left(E_{\infty}, H_{\infty}\right)$ the corresponding far field patterns. Note that  $E_{\infty}$ determines $H_{\infty}$ as well.
\begin{theorem}[Uniqueness of the shape]\label{inverse_sca}
Suppose that $\Omega_1$ and $\Omega_2$ are two dielectric scatterers with $C^1$-smooth boundaries and connected complements. For $j=1,2$, let $\varepsilon_j$, $\mu_j$, and $V_j$ denote the electric permittivity, magnetic permeability, and constant Minkowski constitutive vector in $\Omega_j$, while $\varepsilon_0$ and $\mu_0$ are the common exterior parameters. Fix some $p>3$. Assume that each static contrast is nontrivial, that its leading reflected coefficient for the chosen polarization is uniformly nonzero along the boundary, that $0\leq\beta_j<1$, and that the corresponding static operator $K_{{\rm st},j}$ is invertible on $L^p_*(\Omega_j)^6$. Let $\eta_{V_j}$ be the explicit $V_j$-dependent perturbation parameter in \eqref{regimes-2}, with the parameters of $\Omega_j$.

There exists $\eta_{\rm inv}\in(0,1)$, depending only on the static configurations, the frequency, and the prescribed geometric and subcritical a priori bounds, such that the following holds whenever
\begin{equation}\label{inverse-smallness-condition}
 \max\{\eta_{V_1},\eta_{V_2}\}\leq\eta_{\rm inv}.
\end{equation}
If, at a fixed frequency $\omega>0$, the electric far-field patterns generated by the incident waves \eqref{E_inci}--\eqref{H_inci} coincide for every incident direction $d\in\mathbb S^2$ and every observation direction $\theta\in\mathbb S^2$, for one fixed nonzero polarization seed $\boldsymbol p$, then $\Omega_1=\Omega_2$.
\end{theorem}
\noindent For the electric point-source normalization used below, the reflected-coefficient condition is equivalent to $\varepsilon_j\neq\varepsilon_0$ as the leading coefficient is $\frac{\varepsilon_j-\varepsilon_0}{\varepsilon_j+\varepsilon_0}(I+\nu\otimes\nu)\boldsymbol p/(32\pi)$, as proved in Lemma~\ref{blowup}. Thus the theorem does not include a permeability-only static contrast with vanishing electric leading coefficient.

\noindent There is a large literature on inverse electromagnetic scattering by stationary inhomogeneities; see, for instance, the monographs \cite{MR2193218, MR0700400, MR1853728, MR4385553, MR4651388}. Classical and subsequent results on inverse electromagnetic scattering for penetrable isotropic and anisotropic media, including uniqueness and reconstruction of interfaces and supports, can be found in \cite{ColtonKress1990,ColtonPotthast1999,CakoniColton2003,KarSini2014Maxwell,KuanLinSini2015} and the references therein. More general bi-anisotropic constitutive laws have also been considered in inverse electromagnetic scattering; see, for instance, \cite{Nguyen2019Bianisotropic}. Related stationary inverse transmission results are given in \cite{MR1642603}. Time-dependent moving obstacles are considered in \cite{MR1115177} and the references therein. In the present setting, the inverse problem also has a direct effective-medium interpretation. The parameters of a synthetic moving medium are not measured independently of its electromagnetic response. For instance, in the recent experiment \cite{Yang2026}, the moving coupling and the corresponding effective velocity are retrieved from reflection and transmission data. For a bounded three-dimensional sample, a first question is therefore to determine the spatial support of the effective moving-type response before attempting the simultaneous recovery of all constitutive parameters. Theorem~\ref{inverse_sca} addresses precisely this geometric step, namely full-aperture electric far-field data determine $\Omega$ in the perturbative regime \eqref{inverse-smallness-condition}, while neither the direction nor the magnitude of $V$ is assumed to be known. 
\bigskip

\noindent In a forthcoming work, we develop a nonperturbative well-posedness theory under more general constitutive assumptions and investigate the associated inverse scattering problem, including the simultaneous identification of the constitutive vector \(V\) and the material parameters \(\varepsilon\) and \(\mu\) within suitable classes of spatially varying Minkowski-type media.

\subsection{Organization of the article}
We first recall the function spaces and boundary integral operators used throughout the work in Section~\ref{sec2}.  In Section~\ref{sec3}, we reduce the scattering problem to the corresponding Lippmann--Schwinger system, and in Section~\ref{sec5} we establish its inversion and the Born expansion.  Section~\ref{sec6} returns to the original transmission problem and proves the forward well-posedness result.  Finally, Section~\ref{sec7} proves the shape-identifiability result. Appendix~\ref{app:spacetime} gives a formal derivation of the Minkowski-type constitutive law from a traveling space--time modulation.

\subsection*{Acknowledgements} M. Kar was supported by MATRICS grant (MTR/2019/001349) of SERB. M. Sini was partially supported by the Austrian Science Fund (FWF): P 32660.

\section{Preliminaries}\label{sec2}
Recall that a domain $\Omega \subset \mathbb{R}^3$ is said to be $C^1$-smooth domain if its boundary can be locally described by graphs of $C^1$-smooth functions in appropriate coordinate systems. Here $\Gamma_{\pm}(x)$ denote the so-called nontangential approach regions, i.e. the interiors of the two components (in $\Omega$ and $\mathbb{R}^3\setminus\overline{\Omega}$, respectively) of a regular family of circular, doubly truncated cones $\{\Gamma(x) ; x \in \partial \Omega\}$, with vertex at $x$, as defined in \cite{MR501367} (for $C^1$ domains). 
More details of these and related topics can be found for example in \cite{MR1438894}, \cite{MR501367}, \cite{MR769382}. 
For each $1<p<\infty$, we define the following function spaces
\[
L_{\tan }^p(\partial\Omega)
= \left\{{f} \in L^p(\partial \Omega) ;\langle \nu, f\rangle=0  \ \text{a.e. on}\ \partial \Omega\right\}
\]
and
\[
W^{1,p}(\partial \Omega):=\left\{{f} \in L^p(\partial \Omega) ;\nabla_{\tan }f \in L^p(\partial \Omega)\right\}
\]
where $\nabla_{\tan }=-\nu \times(\nu \times \nabla)$ stands for the tangential gradient. Following \cite{MR1438894, MR1317629}, we introduce the surface divergence operator
$$\operatorname{Div}: L_{\tan }^p(\partial \Omega) \rightarrow W^{-1,p}(\partial \Omega):=\left(W^{1,q}(\partial \Omega)\right)^*, \quad \frac{1}{p} + \frac{1}{q} = 1$$ by requiring 
$$\langle\operatorname{Div}f,g\rangle=-\int_{\partial \Omega}\left\langle f, \nabla_{\tan } g\right\rangle d \sigma$$
for each $f \in L_{\tan }^p(\partial \Omega)$, and $g \in W^{1,q}(\partial \Omega)=\left(W^{-1,p}(\partial \Omega)\right)^*$. We also define
$$L_{\tan }^{p, \operatorname{Div}}(\partial \Omega):=\left\{f \in L_{\tan }^p(\partial \Omega) ; \operatorname{Div} f \in L^p(\partial \Omega)\right\}$$
equipped with the norm
$$\|f\|_{L_{\tan }^{p, \operatorname{Div}}(\partial \Omega)}:=\|f\|_{L^p(\partial \Omega)}+\|\operatorname{Div} f\|_{L^p(\partial \Omega)}.$$
 We also introduce
$$ W^{1,p}_{loc}(\operatorname{curl}, \mathbb{R}^3)=\left\{v \in L_{loc}^p(\mathbb{R}^3), \nabla \times v \in L_{loc}^p(\mathbb{R}^3)\right\}.
$$
For a bounded $C^1$ domain, the tangential trace $u\mapsto \nu\times u$ extends continuously from smooth fields to the corresponding $W^{1,p}(\curl,\Omega)$ trace space; the surface-divergence identity $\operatorname{Div}(\nu\times u)=-\nu\cdot\curl u$ holds in the distributional sense.  These standard trace facts are used in the transmission and representation formulas below; see \cite{MR1438894, MR1317629}.
The notation $W^{1,p}(\curl,\Omega)$ denotes the graph space with norm
$\|u\|_{L^p(\Omega)}+\|\curl u\|_{L^p(\Omega)}$.
\[
 \mathcal X_p(\partial\Omega):=
 \gamma_tW^{1,p}(\curl,\Omega),\qquad \gamma_tu:=\nu\times u,
\]
with its quotient norm. Its elements are tangential boundary distributions. For a smooth vector field $\psi$ and $u\in W^{1,p}(\curl,\Omega)$, their action is determined by
\[
 \langle\gamma_tu,\psi|_{\partial\Omega}\rangle
 =\int_\Omega\big(\curl u\cdot\psi-u\cdot\curl\psi\big)\,dx.
\]
In particular, convergence in the curl graph norm gives convergence of tangential traces in $\mathcal X_p$.

To keep the two levels of notation distinct, $B$ and $D$ always denote the magnetic induction and electric displacement, respectively.  The symbols $\mT,\mB,\mC,\mD$ denote pointwise matrices acting on $\mathbb C^6$, whereas calligraphic symbols such as $\mathcal V_{\omega}$ and $G_j$ denote integral operators on function spaces.

We shall define potential operators and discuss their properties which will be useful in our analysis. Before that, we start with the fundamental solutions for Helmholtz equation. For each $k \in \mathbb{C}$, we let $\Phi_k$ stand for the standard radial fundamental solution for the Helmholtz operator $\Delta+k^2$ in $\mathbb{R}^3$,
\begin{equation}\label{Fundamental_sol}
    \Phi_k(x):=\frac{e^{i k|x|}}{4 \pi|x|}, \quad x \neq 0.
\end{equation}
Our convention is $(\Delta+k^2)\Phi_k=-\delta_0$; in particular, $\Phi_0$ is the fundamental solution of $-\Delta$ in $\mathbb R^3$.
The single-layer acoustic potential operator with density $f$ is defined by
$$
\mathcal{S}_k f(x):=\int_{\partial \Omega} \Phi_k(x-y) f(y) d \sigma(y), \quad x \in \mathbb{R}^3 \backslash \partial \Omega .
$$
For any $f\in L^p(\partial \Omega)$, the operator $\mathcal{S}_k f$ satisfies the Helmholtz equation $\left(\Delta+k^2\right) \mathcal{S}_k f=0$ in $\mathbb{R}^3 \backslash \partial \Omega$. The interior / exterior nontangential boundary traces of $\mathcal{S}_k f$ are given by
$$
\lim _{\substack{x \rightarrow y \\ x \in \Gamma_{+}(y)}} \mathcal{S}_k f(x)=\lim _{\substack{x \rightarrow y \\ x \in \Gamma_{-}(y)}} \mathcal{S}_k f(x)=S_k f(y), \quad y \in \partial \Omega,
$$
where
$$
S_k f(y):=\frac{1}{4 \pi} \int_{\partial \Omega} \frac{e^{i k|x-y|}}{|x-y|} f(x) d \sigma(x), \quad y \in \partial \Omega .
$$
For almost any $y \in \partial \Omega$,
$$
\lim _{\substack{x \rightarrow y \\ x \in \Gamma_{\pm}(y)}} \frac{\partial \mathcal{S}_k f}{\partial \nu}(x):=\lim _{\substack{x \rightarrow y \\ x \in \Gamma_{\pm}(y)}}\left\langle \nu(y), \nabla \mathcal{S}_k f(x)\right\rangle=:\left(\pm \frac{1}{2} I+K_k^*\right) f(y),
$$
where $K_k^*$ is the formal transpose of the principal value integral operator
$$
K_k f(y):=-\text { p.v. } \frac{1}{4 \pi} \int_{\partial \Omega} \frac{\langle \nu(x), x-y\rangle}{|x-y|^3} e^{i k|x-y|}(1-i k|x-y|) f(x) d \sigma(x),
$$
the boundary double-layer operator. The corresponding potential is
\[
 \mathcal K_kf(x):=\int_{\partial\Omega}
 \partial_{\nu(y)}\Phi_k(x-y)f(y)\,d\sigma(y),
 \qquad x\notin\partial\Omega.
\]
Its nontangential limits satisfy
$$
\lim _{\substack{x \rightarrow y \\ x \in \Gamma_{\pm}(y)}} \mathcal{K}_k f(x)=\left(\mp \frac{1}{2} I+K_k\right) f(y),
$$
at almost any $y \in \partial \Omega$.
For $1<p<\infty$, we have
$$
\lim _{\substack{x \rightarrow y \\ x \in \Gamma_{\pm}(y)}} \operatorname{div} \mathcal{S}_k f(x)=\pm \frac{1}{2}\langle \nu, f\rangle(y)+\text { p.v. } \int_{\partial \Omega} \operatorname{div}_y\left\{\Phi_k(y-z) f(z)\right\} d \sigma,
$$
and
$$
\lim _{\substack{x \rightarrow y \\ x \in \Gamma_{\pm}(y)}} \operatorname{curl} \mathcal{S}_k f(x)=\pm \frac{1}{2}(\nu \times f)(y)+\text { p.v. } \int_{\partial \Omega} \operatorname{curl}_y\left\{\Phi_k(y-z) f(z)\right\} d \sigma
$$
holds for any $f\in L^p(\partial \Omega)^3$ and for almost any $y \in \partial \Omega$,
see \cite{MR501367, MR672839}.
We could also define the magnetic dipole operator by
$$
M_k f:=\nu \times\left(\text { p.v. curl }\,\mathcal S_k f\right)
$$
and it follows that 
$$
\lim _{\substack{x \rightarrow y \\ x \in \Gamma_{\pm}(y)}} \nu(y) \times \operatorname{curl} \mathcal{S}_k f(x)=\left(\mp \frac{1}{2} I+M_k\right) f(y)
$$
holds for  $f \in L_{\tan }^p(\partial \Omega)$ and
at almost any $y \in \partial \Omega$, see for instance \cite{MR672839}.
Finally, we recall the Newtonian potential
\begin{equation}\label{Newton_Hel}
   L_k f(x):=\int_{\Omega} \Phi_k(x-y) f(y) d y, \quad x \in \mathbb{R}^3
\end{equation}
which is well defined for $f\in L^p(\Omega), p>1$. The induced Newtonian operator $L_k$ is actually well defined and bounded from $L^p(\Omega)$ to $W_{loc}^{2, p}(\mathbb{R}^3), ~p>1$.

\section{Reduction to the Lippmann--Schwinger System}\label{sec3}
In this section, we derive the Lippmann--Schwinger system associated with problem \eqref{model_EQ}; its inversion is treated in Section~\ref{sec5}. 
We will first write down our model as a perturbation of classical Maxwell system. We first start with fundamental solutions for classical Maxwell equations.
Remark from \eqref{model_EQ} that, the fields $E^{s}$ and $H^{s}$ satisfy the classical Maxwell system 
\[
\begin{array}{l}
\curl\; \begin{pmatrix}
 E^{s} \\ \\
 H^{s}
\end{pmatrix}
=
-i\omega
\tilde{A}^{-}\begin{pmatrix}
E^{s}\\ \\
 H^{s}
 \end{pmatrix}
 \end{array}
\]
in the domain $\mathbb{R}^3 \backslash \overline{\Omega}$, where 
\[
\tilde{A}^{-} =
\begin{pmatrix}
0 & -\mu_{0} \\ \\
\varepsilon_{0} & 0 
\end{pmatrix}
\]
and 
\begin{equation}\label{exterior_problem}
\begin{cases}\curl \ E^{s}=i\omega\mu_{0} H^{s} & \text { in } \mathbb{R}^3 \backslash \overline{\Omega} \\ \curl \ H^{s}=-i\omega\epsilon_{0} E^{s} & \text { in } \mathbb{R}^3 \backslash \overline{\Omega}.
\end{cases}
\end{equation}
It is well known that the outgoing fundamental solution of the Maxwell system \eqref{exterior_problem} is 
\begin{equation}\label{fund_soll}
  G = \omega \varepsilon_{0} \mu_{0}\left(\begin{array}{cc}
1+\frac{\nabla \nabla \cdot}{\omega^2 \varepsilon_{0} \mu_{0}} & \frac{i}{\omega \varepsilon_{0}} \nabla \times \\
-\frac{i}{\omega \mu_{0}} \nabla \times & 1+\frac{\nabla \nabla .}{\omega^2 \varepsilon_{0} \mu_{0}}
\end{array}\right) \Phi_{k_{0}},  
\end{equation}
where $\Phi_{k_{0}}$ is the radial fundamental solution for the Helmholtz operator $\Delta+k_{0}^2$ in $\mathbb{R}^3$.
If we define the operator $\mathcal{L}_0$ as
\[
\mathcal{L}_0 = \begin{pmatrix}
0 & \frac{i}{\varepsilon_{0}}\nabla\times \\ \\
-\frac{i}{\mu_{0}} \nabla\times & 0 
\end{pmatrix},
\]
then the fundamental solution $G$ satisfies
\[
(\mathcal{L}_0 - \omega)G = \delta I
\]
where $\delta$ is the Dirac distribution and $I$ is the identity matrix. Now we want to write down the differential operator appearing in equation \eqref{model_EQ} in the domain $\Omega$ as a perturbation of the differential operator $\mathcal{L}_0 - \omega$. To do that, we first write the equation \eqref{eq_matrix2} in the domain $\Omega$ as
\begin{equation}\label{Omega_1}
    -\nabla\times E + \mu V\times\nabla\times H = - i\omega\frac{1}{c^2}V\times E - i\omega\mu H
\end{equation}
and 
\begin{equation}\label{Omega_2}
     \varepsilon V\times\nabla\times E + \nabla\times H = - i\omega\varepsilon E +i\omega \frac{1}{c^2}V\times H.
\end{equation}
Multiplying by $\frac{i}{\mu_{0}}$ to the equation \eqref{Omega_1} and by $\frac{i}{\varepsilon_{0}}$ to the equation \eqref{Omega_2}, we obtain
\begin{equation}\label{Omega_3}
    -\frac{i}{\mu_{0}}\nabla\times E + \frac{i}{\mu_{0}}\mu V\times\nabla\times H = - \frac{i}{\mu_{0}}i\omega\frac{1}{c^2}V\times E - \frac{i}{\mu_{0}}i\omega\mu H
\end{equation}
and 
\begin{equation}\label{Omega_4}
    \frac{i}{\varepsilon_{0}} \varepsilon V\times\nabla\times E + \frac{i}{\varepsilon_{0}}\nabla\times H = - \frac{i}{\varepsilon_{0}}i\omega\varepsilon E + \frac{i}{\varepsilon_{0}}i\omega \frac{1}{c^2}V\times H.
\end{equation}
Simplifying further the above two equations, we write
\begin{equation}
 \begin{pmatrix}
1 & -\mu_{0}\frac{\varepsilon}{\varepsilon_{0}} V\times \\ \\
\epsilon_{0} \frac{\mu}{\mu_{0}}V\times & 1 
\end{pmatrix}
\begin{pmatrix}
 \frac{i}{\varepsilon_{0}}\nabla\times H \\ \\
 -\frac{i}{\mu_{0}}\nabla\times E
\end{pmatrix}
=
\begin{pmatrix}
1 & -\frac{1}{c^2}\frac{1}{\varepsilon_{0}}\frac{\mu_{0}}{\mu}V\times \\ \\
\frac{1}{c^2}\frac{1}{\mu_{0}}\frac{\varepsilon_{0}}{\varepsilon}V\times & 1 
\end{pmatrix}
\begin{pmatrix}
\omega\frac{\varepsilon}{\varepsilon_{0}} E\\ \\
\omega\frac{\mu}{\mu_{0}} H
\end{pmatrix}
\end{equation}
and this implies
\begin{equation}\label{eq_matrix5}
\begin{pmatrix}
 \frac{i}{\varepsilon_{0}}\nabla\times H \\ \\
 -\frac{i}{\mu_{0}}\nabla\times E
\end{pmatrix}
=
 \begin{pmatrix}
1 & -\mu_{0}\frac{\varepsilon}{\varepsilon_{0}} V\times \\ \\
\varepsilon_{0} \frac{\mu}{\mu_{0}}V\times & 1 
\end{pmatrix}^{-1}
\begin{pmatrix}
1 & -\frac{1}{c^2}\frac{1}{\varepsilon_{0}}\frac{\mu_{0}}{\mu}V\times \\ \\
\frac{1}{c^2}\frac{1}{\mu_{0}}\frac{\varepsilon_{0}}{\varepsilon}V\times & 1 
\end{pmatrix}
\begin{pmatrix}
\omega\frac{\varepsilon}{\varepsilon_{0}} E\\ \\
\omega\frac{\mu}{\mu_{0}} H
\end{pmatrix}.
\end{equation}
Set $K:=V\times$ and introduce the pointwise matrices
\begin{equation}\label{def-TCD}
\mT:=
\begin{pmatrix}
0&\mu_0\dfrac{\varepsilon}{\varepsilon_0}K\\[1mm]
-\varepsilon_0\dfrac{\mu}{\mu_0}K&0
\end{pmatrix},
\qquad
\mC:=
\begin{pmatrix}
0&-\dfrac{1}{c^2}\dfrac{1}{\varepsilon_0}\dfrac{\mu_0}{\mu}K\\[2mm]
\dfrac{1}{c^2}\dfrac{1}{\mu_0}\dfrac{\varepsilon_0}{\varepsilon}K&0
\end{pmatrix},
\qquad
\mD:=I_{6\times6}.
\end{equation}
Then the first matrix on the right-hand side of \eqref{eq_matrix5} is $(I-\mT)^{-1}$ and the second one is $\mC+\mD$.  The special form of $K$ gives substantially more information than an entrywise estimate.  Indeed,
\begin{equation}\label{T-powers}
 \mT^2=-\varepsilon\mu
 \begin{pmatrix}K^2&0\\0&K^2\end{pmatrix},
 \qquad
 \mT^3=\beta^2\mT,
 \qquad \beta=|V|\sqrt{\varepsilon\mu}=\frac{|V|}{c_{\Omega}}.
\end{equation}
Moreover, in the weighted norm \eqref{weighted-norm},
\begin{equation}\label{T-weighted-norm}
 \|\mT F\|_{p,*}\leq\beta\|F\|_{p,*},
 \qquad
 \|\mT\|_{*}=\beta.
\end{equation}
To verify this, conjugate $\mT$ by
$\operatorname{diag}(\sqrt b I_{3\times3},\sqrt a I_{3\times3})$ and use $ab=\varepsilon\mu$; the resulting off-diagonal blocks are $\pm\sqrt{\varepsilon\mu}\,K$.

It follows from \eqref{T-powers} that, for $\beta\neq1$,
\begin{equation}\label{exact-IminusT}
 (I-\mT)^{-1}
 =I+\frac{\mT+\mT^2}{1-\beta^2}.
\end{equation}
In the subluminal regime $\beta<1$, this inverse is also represented by the convergent series
\begin{equation}\label{B-series}
 (I-\mT)^{-1}=I+\mB,
 \qquad
 \mB:=\sum_{j=1}^{\infty}\mT^j
 =\frac{\mT+\mT^2}{1-\beta^2}.
\end{equation}
Thus the former entrywise condition on $\mT$ is replaced by the sharp constitutive convergence condition $|V|<c_\Omega$ for this pointwise Neumann expansion.

A second useful cancellation follows directly from \eqref{def-TCD}:
\begin{equation}\label{CplusT-cancellation}
 \mC+\mT
 =\left(1-\frac{c_{\Omega}^2}{c^2}\right)\mT
 =\vartheta\mT,
\end{equation}
where the last equality uses $c_{\Omega}<c$.  Consequently,
\begin{align}
 \mC+\mB\mC+\mB\mD
 &= (I-\mT)^{-1}(\mC+\mT)\label{motion-matrix-factorization}\\
 &=\vartheta\sum_{j=1}^{\infty}\mT^j
 =\vartheta\frac{\mT+\mT^2}{1-\beta^2}.
 \label{motion-matrix-exact}
\end{align}
Let
\begin{equation}\label{def-Q}
 Q:=
 \begin{pmatrix}
 r_{\varepsilon}I_{3\times3}&0\\
 0&r_{\mu}I_{3\times3}
 \end{pmatrix}.
\end{equation}
Using \eqref{eq_matrix5}--\eqref{motion-matrix-factorization}, we obtain
\begin{align}
(\mathcal L_0-\omega)\binom EH
&=\omega\mathcal M\binom EH,
\label{perturbed-Maxwell-system}
\end{align}
where the effective contrast matrix is
\begin{equation}\label{effective-contrast-exact}
 \mathcal M
 =\underbrace{
 \begin{pmatrix}
 (r_{\varepsilon}-1)I_{3\times3}&0\\
 0&(r_{\mu}-1)I_{3\times3}
 \end{pmatrix}}_{\mathcal M_{\rm st}}
 +\underbrace{(I-\mT)^{-1}(\mC+\mT)Q}_{\mathcal M_{\rm mov}}.
\end{equation}
Equivalently, by \eqref{motion-matrix-exact},
\begin{equation}\label{effective-contrast-expanded}
 \mathcal M_{\rm mov}
 =\vartheta\frac{\mT+\mT^2}{1-\beta^2}Q
 =\vartheta\sum_{j=0}^{\infty}\mT^{j+1}Q.
\end{equation}
Since $-K^2=|V|^2\Pperp$, where $\Pperp$ is the orthogonal projection onto $V^{\perp}$, this identity may also be written in the explicit block form
\begin{equation}\label{motion-matrix-block-form}
 \mathcal M_{\rm mov}
 =\frac{\vartheta}{1-\beta^2}
 \begin{pmatrix}
  \beta^2 r_{\varepsilon}\Pperp
  &\mu_0r_{\varepsilon}r_{\mu}K\\[1mm]
  -\varepsilon_0r_{\varepsilon}r_{\mu}K
  &\beta^2r_{\mu}\Pperp
 \end{pmatrix}.
\end{equation}
Thus the first-order contribution is carried by the skew-symmetric cross-product blocks, while the diagonal correction is quadratic in $|V|$. The decomposition \eqref{effective-contrast-exact} separates the usual static dielectric contrast from the contribution induced by the motion.

As $(\mathcal{L}_0 - \omega)\begin{pmatrix}
 E^i\\ \\
 H^i
\end{pmatrix}=0$, then we get 
\begin{equation}
(\mathcal{L}_0 - \omega)\begin{pmatrix}
 E-E^i\\ \\
 H-H^i
\end{pmatrix}=\omega\mathcal{M}\begin{pmatrix}
 E \\ \\
 H
\end{pmatrix}.
\end{equation} 
Since $E-E^i$ and $H-H^i$ satisfy the Silver-M\"uller radiation conditions, then this last system implies (\ref{Lipp_s}).
\section{Inversion of the Lippmann--Schwinger system of equations}\label{sec5}
Let
\[
 U:=\binom EH,\qquad U^i:=\binom{E^i}{H^i}.
\]
With the notation of \eqref{effective-contrast-exact}, define
\begin{equation}\label{GM-definition}
 G_{\mathcal M}U:=\mathcal V_{\omega}(\mathcal M U),
 \qquad
 G_{\rm st}U:=\mathcal V_{\omega}(\mathcal M_{\rm st}U).
\end{equation}
The exact expansion \eqref{effective-contrast-expanded} yields
\begin{equation}\label{G_j-equation}
 G_jU:=\mathcal V_{\omega}\big(\vartheta\mT^{j+1}QU\big),
 \qquad j\geq0,
\end{equation}
and hence
\begin{equation}\label{GM-decomposition}
 G_{\mathcal M}=G_{\rm st}+G_{\rm mov},
 \qquad
 G_{\rm mov}:=\sum_{j=0}^{\infty}G_j.
\end{equation}
The Lippmann--Schwinger equation becomes
\begin{equation}\label{LS-operator-form}
 (I-G_{\mathcal M})U=U^i,
 \qquad
 I-G_{\mathcal M}=K_{\rm st}-G_{\rm mov},
\end{equation}
where $K_{\rm st}=I-G_{\rm st}$ is the static operator introduced in \eqref{static-resolvent}.

The next theorem gives the exact spectral solvability criterion associated with this factorization, the precise criterion for convergence of the motion Born series, and a directly verifiable sufficient estimate.
\begin{theorem}\label{opera_est}
Let $1<p<\infty$ and let $\Omega$ be a bounded $C^1$ domain. Assume that $0\leq\beta<1$ and that $K_{\rm st}$ is invertible on $L^p_*(\Omega)^6$. Then the series defining $G_{\rm mov}$ converges in operator norm and
\begin{align}
 \|G_j\|_{*}
 &\leq \mathfrak C_{p,\omega,\Omega}^{*}\,\vartheta r_*\beta^{j+1},
 \qquad j\geq0,\label{Gj-sharp-bound}\\
 \|G_{\rm mov}\|_{*}
 &\leq \mathfrak C_{p,\omega,\Omega}^{*}\,\vartheta r_*
 \frac{\beta}{1-\beta}.
 \label{Gmov-sharp-bound}
\end{align}
Set
\begin{equation}\label{AV-definition}
 \mathcal A_V:=K_{\rm st}^{-1}G_{\rm mov}.
\end{equation}
Then
\begin{equation}\label{exact-spectral-solvability}
 I-G_{\mathcal M}\ \text{is invertible}
 \quad\Longleftrightarrow\quad
 1\notin\sigma(\mathcal A_V).
\end{equation}
Moreover, the motion Born series converges in operator norm precisely when
\begin{equation}\label{spectral-Born-condition}
 r_{\rm sp}(\mathcal A_V)<1,
\end{equation}
and in that case
\begin{equation}\label{motion-Born-series}
 U=\sum_{n=0}^{\infty}
 \mathcal A_V^nK_{\rm st}^{-1}U^i.
\end{equation}
A directly verifiable sufficient condition for \eqref{spectral-Born-condition} is
\begin{equation}\label{exact-motion-Neumann-condition}
 \rho_V:=\|\mathcal A_V\|_{*}<1.
\end{equation}
In particular, Assumption~\ref{assum}(iii) implies \eqref{exact-motion-Neumann-condition}. Under the exact solvability condition \eqref{exact-spectral-solvability}, one has
\begin{equation}\label{general-solution-weighted-bound}
 \|U\|_{p,*}
 \leq \kappa_{\rm st}\,\|(I-\mathcal A_V)^{-1}\|_*
 \|U^i\|_{p,*}.
\end{equation}
Under \eqref{exact-motion-Neumann-condition}, the first inequality below holds; the second additionally assumes $\eta_V<1$:
\begin{equation}\label{solution-weighted-bound}
 \|U\|_{p,*}
 \leq \frac{\kappa_{\rm st}}{1-\rho_V}\,\|U^i\|_{p,*}
 \leq \frac{\kappa_{\rm st}}{1-\eta_V}\,\|U^i\|_{p,*}.
\end{equation}
If the explicit condition \eqref{explicit-total-smallness} holds, then also
\begin{equation}\label{full-Born-series}
 U=\sum_{n=0}^{\infty}G_{\mathcal M}^nU^i,
 \qquad
 \|U\|_{p,*}\leq
 \frac{1}{1-\mathfrak C_{p,\omega,\Omega}^{*}
 (\delta_{\rm st}+\delta_{\rm mov})}\,\|U^i\|_{p,*}.
\end{equation}
\end{theorem}

\begin{proof}
We first make explicit the norm of the free Maxwell volume potential. Let $F=(F_1,F_2)^{\mathsf T}\in L^p(\Omega)^6$. From \eqref{fund_soll} and \eqref{free-volume-operator}, with $k_0^2=\omega^2\varepsilon_0\mu_0$, one obtains
\begin{equation}\label{volume-potential-components}
 \mathcal V_{\omega}F=
 \binom{
 k_0^2L_{k_0}F_1+\nabla\nabla\!\cdot L_{k_0}F_1
 +i\omega\mu_0\nabla\!\times L_{k_0}F_2
 }{
 -i\omega\varepsilon_0\nabla\!\times L_{k_0}F_1
 +k_0^2L_{k_0}F_2+\nabla\nabla\!\cdot L_{k_0}F_2
 }.
\end{equation}
The classical Newton-potential estimate gives a constant
$C_{p,\Omega,k_0}>0$ such that
\begin{equation}\label{Newton-W2p-estimate}
 \|L_{k_0}f\|_{W^{2,p}(\Omega)}
 \leq C_{p,\Omega,k_0}\|f\|_{L^p(\Omega)}.
\end{equation}
Indeed, $(e^{ik_0r}-1)/(4\pi r)=ik_0/(4\pi)-k_0^2r/(8\pi)+O(r^2)$ as $r\to0$. Thus the second derivatives of the difference kernel have only an integrable $O(r^{-1})$ singularity on bounded sets. The Hessian of $L_0$ is controlled by the Calder\'on--Zygmund estimate; see, for example, \cite[Theorem~9.9]{MR1814364}. Thus \eqref{Newton-W2p-estimate} also holds componentwise for vector-valued functions. The distributional convention includes the local term
\[
 \partial_i\partial_j\Phi_0
 =\operatorname{p.v.}\frac{3x_ix_j-|x|^2\delta_{ij}}{4\pi|x|^5}
 -\frac13\delta_{ij}\delta_0.
\]
The same local term occurs in $\partial_i\partial_j\Phi_{k_0}$; it must not be omitted when writing the volume potential as a principal-value integral.

Applying \eqref{Newton-W2p-estimate} to \eqref{volume-potential-components} yields
\begin{align}
 \| (\mathcal V_{\omega}F)_1\|_{L^p}
 &\leq C_{p,\Omega,k_0}
 \big((1+k_0^2)\|F_1\|_{L^p}
 +\omega\mu_0\|F_2\|_{L^p}\big),\label{V-first-component}\\
 \| (\mathcal V_{\omega}F)_2\|_{L^p}
 &\leq C_{p,\Omega,k_0}
 \big(\omega\varepsilon_0\|F_1\|_{L^p}
 +(1+k_0^2)\|F_2\|_{L^p}\big).
 \label{V-second-component}
\end{align}
After multiplication by the weights in \eqref{weighted-norm}, these inequalities give
\begin{equation}\label{explicit-volume-bound}
 \mathfrak C_{p,\omega,\Omega}^{*}
 \leq C_{p,\Omega,k_0}
 \max\left\{
 1+k_0^2+\omega\varepsilon_0\sqrt{\frac{a}{b}},
 1+k_0^2+\omega\mu_0\sqrt{\frac{b}{a}}
 \right\}.
\end{equation}
This displays explicitly the dependence on the Newton-potential bound, the frequency, the background coefficients, and the impedance weights.

The diagonal matrix $Q$ commutes with the weights and satisfies
\begin{equation}\label{Q-weighted-bound}
 \|QF\|_{p,*}\leq r_*\|F\|_{p,*}.
\end{equation}
Together with \eqref{T-weighted-norm}, this gives
\[
 \|\vartheta\mT^{j+1}QF\|_{p,*}
 \leq \vartheta r_*\beta^{j+1}\|F\|_{p,*}.
\]
Applying the definition \eqref{volume-operator-constant} of
$\mathfrak C_{p,\omega,\Omega}^{*}$ proves \eqref{Gj-sharp-bound}; summing the geometric series proves \eqref{Gmov-sharp-bound}.

Now factor the full operator in the order
\begin{equation}\label{static-factorization}
 I-G_{\mathcal M}
 =K_{\rm st}(I-\mathcal A_V).
\end{equation}
Since $K_{\rm st}$ is invertible, \eqref{exact-spectral-solvability} follows immediately. The geometric series $\sum_{n\geq0}\mathcal A_V^n$ converges in operator norm if and only if $r_{\rm sp}(\mathcal A_V)<1$: sufficiency follows from the spectral-radius formula, while necessity follows because convergence implies $\mathcal A_V^n\to0$ and hence $\|\mathcal A_V^N\|_*<1$ for some $N$. This proves \eqref{motion-Born-series} under \eqref{spectral-Born-condition}.

The factorization \eqref{static-factorization} also gives \eqref{general-solution-weighted-bound}. If \eqref{exact-motion-Neumann-condition} holds, then
\[
 \|(I-\mathcal A_V)^{-1}\|_*
 \leq\frac{1}{1-\rho_V}.
\]
By \eqref{Gmov-sharp-bound},
\[
 \rho_V
 \leq\kappa_{\rm st}\mathfrak C_{p,\omega,\Omega}^{*}
 \vartheta r_*\frac{\beta}{1-\beta}=\eta_V,
\]
which proves the second inequality in \eqref{solution-weighted-bound} whenever $\eta_V<1$.

Finally,
\[
 \|G_{\rm st}\|_*
 \leq\mathfrak C_{p,\omega,\Omega}^{*}\delta_{\rm st},
 \qquad
 \|G_{\mathcal M}\|_*
 \leq\mathfrak C_{p,\omega,\Omega}^{*}
 (\delta_{\rm st}+\delta_{\rm mov}).
\]
Hence \eqref{explicit-total-smallness} implies the full Neumann expansion \eqref{full-Born-series}.
\end{proof}

\begin{remark}[Dependence on the constitutive vector $V$]\label{rem:analytic-V}
On every set $\beta\leq\beta_*<1$, formula \eqref{effective-contrast-expanded} converges uniformly in operator norm.  Therefore $\mathcal M_{\rm mov}$ is analytic in the components of $V$.  Whenever the operator-norm in \eqref{exact-motion-Neumann-condition} is uniform, the same is true of the total field.  In particular, the first $V$-dependent correction to the static field $U_{\rm st}:=K_{\rm st}^{-1}U^i$ is
\[
 K_{\rm st}^{-1}\mathcal V_{\omega}\big(\vartheta\mT Q U_{\rm st}\big),
\]
and the remainder is $O(\beta^2)$ in $L^p_*(\Omega)^6$ as $\beta\to0$.
\end{remark}

\section{Well-posedness of the forward problem (\ref{model_EQ})}\label{sec6}
In this section, we justify the equivalence with the Lippmann--Schwinger formulation and complete the forward well-posedness proof:
\begin{equation}\label{Lipp_spacetime}
   \begin{pmatrix}
 E-E^{i}\\ \\
 H-H^{i}
\end{pmatrix}(x)
=
\omega
\int_{\Omega}
G(x-y)
\mathcal{M} \begin{pmatrix}
 E(y)\\ \\
 H(y)
\end{pmatrix} dy.
\end{equation}
We base our analysis on \cite{Ammari96time-harmonicelectromagnetic} and \cite{ammari2017doublenegative} which deal with the electromagnetic scattering in chiral media. We first state and prove the following Lemma.
\begin{lemma}\label{representation}
Let $\Omega$ be a bounded $C^1$-smooth domain in $\mathbb{R}^3$ with $1<p<\infty$. Let $u,v\in L^p(\Omega)^3$ be such that $\curl u,\curl v\in L^p(\Omega)^3$. Then at almost any point in $\Omega$ one has
\begin{equation}
\begin{pmatrix}
u(x)\\ \\
 v(x)
\end{pmatrix}=\int_{\Omega} G(x-y)\left(\mathcal{L}_0-\omega\right)\begin{pmatrix}
u(y)\\ \\
 v(y)
\end{pmatrix} d y 
-\begin{pmatrix}
   K_{\Omega} &  \frac{i}{\omega\varepsilon_{0}} D_{\Omega} \\
- \frac{i}{\omega\mu_{0}} D_{\Omega} & K_{\Omega} 
\end{pmatrix}
\begin{pmatrix}
\nu \times u|_{\partial\Omega}\\ \\
 \nu \times v|_{\partial\Omega}
\end{pmatrix},
\end{equation}
where $\partial\Omega$ is the boundary of $\Omega$ and the integral operators $K_{\Omega}, D_{\Omega}$ are defined as
\[
K_{\Omega}f(x) = \curl \mathcal{S}_{k_0}f(x), \ \ x\in \Omega
\]
and 
\[
D_{\Omega}f(x) = \curl\curl \mathcal{S}_{k_0}f(x), \ \ x\in \Omega
\]
initially for tangential $L^p$ densities and, for $f\in\mathcal X_p(\partial\Omega)$, by the duality definition
\[
 (\mathcal S_{k_0}f)_\ell(x)
 :=\langle f,\Phi_{k_0}(x-\cdot)e_\ell\rangle,
 \qquad x\notin\partial\Omega,\quad \ell=1,2,3.
\]
The traces in the representation formula are understood in $\mathcal X_p(\partial\Omega)$.
\end{lemma}
\begin{proof}
We first prove the lemma for all functions $u, v$ that are in $C^{\infty}(\overline{\Omega})$. The first component of the vector $G(x-y)\left(\mathcal{L}_0-\omega\right)\begin{pmatrix}
u(y)\\ \\
 v(y)
\end{pmatrix}$
is 
\[
\omega\varepsilon_{0}\mu_{0}\left((1+\frac{\nabla_{x}\nabla_{x}\cdot}{\omega^2\varepsilon_{0}\mu_{0}})[\Phi_{k_{0}}(-\omega u+\frac{i}{\varepsilon_{0}}\nabla_{y}\times v)]+ (\frac{i}{\omega\varepsilon_{0}}\nabla_x\times)[\Phi_{k_{0}}(-\frac{i}{\mu_{0}}\nabla_{y}\times u - \omega v)] \right)
\]
and the second component
\[
\omega\varepsilon_{0}\mu_{0}\left((-\frac{i}{\omega\mu_{0}}\nabla_x\times)[\Phi_{k_{0}}(\frac{i}{\varepsilon_{0}}\nabla_{y}\times v - \omega u)] + (1+\frac{\nabla_{x}\nabla_{x}\cdot}{\omega^2\epsilon_{0}\mu_{0}})[\Phi_{k_{0}}(-\omega v-\frac{i}{\mu_{0}}\nabla_{y}\times u)]\right).
\]
We now denote $u^{1st}$ as the first component of the vector $G(x-y)\left(\mathcal{L}_0-\omega\right)\begin{pmatrix}
u(y)\\ \\
 v(y)
\end{pmatrix}$ that contains only $u$ term. Therefore using the identities 
\begin{equation}\label{identities}
\begin{aligned}
&\curl \ L_k u =L_k(\curl u)-\mathcal{S}_k(\nu \times u), \\
 &   \curl\curl = -\Delta + \nabla\nabla\cdot ,\\
 & \Delta_x\Phi_{k_{0}} + k_{0}^{2}\Phi_{k_{0}} = -\delta,
\end{aligned}
\end{equation}
we obtain
\[
\begin{aligned}
\int_{\Omega}u^{1st}(y) dy
&=-\omega^2\varepsilon_{0}\mu_{0}\int_{\Omega}\Phi_{k_{0}}(x-y)u(y)dy -\nabla_x\nabla_x\cdot\int_{\Omega}\Phi_{k_{0}}(x-y)u(y)dy
\\&\qquad+ \nabla_x\times\int_{\Omega}\Phi_{k_{0}}(x-y)(\nabla_y\times u)(y)dy\\
& = -k_{0}^2L_{k_{0}}u(x) - \nabla_x\nabla_x\cdot L_{k_{0}}u(x) + \curl_x L_{k_{0}}(\curl_x u)(x)\\
& = -k_{0}^2L_{k_{0}}u(x) - \curl\curl L_{k_{0}}u(x) - \Delta L_{k_{0}}u(x) + \curl L_{k_{0}}(\curl u)(x)\\
& = u(x) - \curl\curl L_{k_{0}}u(x)  + \curl L_{k_{0}}(\curl u)(x)\\
& = u(x) + \curl\mathcal{S}_{k_{0}}(\nu\times u).
\end{aligned}
\]
Similarly, if we denote $v^{1st}$ as the first component of the vector $G(x-y)\left(\mathcal{L}_0-\omega\right)\begin{pmatrix}
u(y)\\ \\
 v(y)
\end{pmatrix}$ that contains only $v$ term, then using the identities \eqref{identities}, we have
\[
\begin{aligned}
\int_{\Omega}v^{1st}(y)\,dy
&=i\omega\mu_0\int_{\Omega}\Phi_{k_0}(x-y)(\nabla_y\times v)(y)\,dy\\
&\quad+\frac{i}{\omega\varepsilon_0}\nabla_x\nabla_x\!\cdot
  \int_{\Omega}\Phi_{k_0}(x-y)(\nabla_y\times v)(y)\,dy\\
&\quad-i\omega\mu_0\nabla_x\times
  \int_{\Omega}\Phi_{k_0}(x-y)v(y)\,dy\\
&=i\omega\mu_0L_{k_0}(\curl v)(x)
 +\frac{i}{\omega\varepsilon_0}\nabla_x\nabla_x\!\cdot L_{k_0}(\curl v)(x)\\
&\quad-i\omega\mu_0\curl L_{k_0}v(x)\\
&=-\frac{i}{\omega\varepsilon_0}\curl v(x)
 +\frac{i}{\omega\varepsilon_0}\curl\curl L_{k_0}(\curl v)(x)\\
&\quad-i\omega\mu_0\curl L_{k_0}v(x)\\
&=\frac{i}{\omega\varepsilon_0}
  \curl\curl\mathcal S_{k_0}(\nu\times v)(x).
\end{aligned}
\]
On the other hand, we now denote $u^{2nd}$ as the second component of the vector $G(x-y)\left(\mathcal{L}_0-\omega\right)\begin{pmatrix}
u(y)\\ \\
 v(y)
\end{pmatrix}$ that contains only $u$ term and defined by
\[
u^{2nd}(y) =i\omega\varepsilon_{0}\nabla_x\times(\Phi_{k_{0}}u(y)) - i\omega\varepsilon_{0}\left( 1 + \frac{\nabla_x\nabla_x\cdot}{\omega^2\varepsilon_{0}\mu_{0}}\right)\Phi_{k_{0}}(\nabla_y\times u(y)).
\]
Also, we denote $v^{2nd}$ as the second component of the vector $G(x-y)\left(\mathcal{L}_0-\omega\right)\begin{pmatrix}
u(y)\\ \\
 v(y)
\end{pmatrix}$ that contains only $v$ term and defined by
\[
v^{2nd}(y)= \nabla_x\times (\Phi_{k_{0}}(\nabla_y\times v(y)))-\omega^2\varepsilon_{0}\mu_{0}\Phi_{k_{0}}v(y) - \nabla_x\nabla_x\cdot (\Phi_{k_{0}} v(y)).
\]
Then similar computations as above imply 
\[
\int_{\Omega}u^{2nd}(y)dy = -\frac{i}{\omega\mu_{0}}\curl\, \curl\, \mathcal{S}_{k_{0}}(\nu\times u).
\]
and
\[
\int_{\Omega}v^{2nd}(y)dy = v(x)+ \curl\, \mathcal{S}_{k_{0}}(\nu\times v).
\]
Therefore the result holds true for any  $u, v$ that are in $C^{\infty}(\overline{\Omega})$. To prove the general case, we apply the convolution and mollification argument as in \cite[Theorem 3.2]{MR1438894}. Let us first
extend $u$ and $\operatorname{curl} u$ with zero outside $\Omega$. Then we claim that there exists a sequence $\left(u_j\right)_j$ of functions in $C^{\infty}(\overline{\Omega})$ with $u_j$, $\operatorname{curl} u_j$ converging to $u$, $\operatorname{curl} u$ respectively in the $L^p(\Omega)$ norm. Without loss of generality using a partition of unity, we assume that supp $u \cap \partial \Omega$ lies in a coordinate patch of $\partial \Omega$. Therefore we could construct a sufficiently small truncated exterior cone $\Gamma$ with vertex at the origin of $\mathbb{R}^3$ such that
\begin{equation}\label{mollification}
    \Gamma+(\partial \Omega \cap \operatorname{supp} u) \subseteq \mathbb{R}^3 \backslash \overline{\Omega}.
\end{equation}
Let $\varphi$ be a smooth, compactly supported function in $\mathbb{R}^3$ having integral one such that $\operatorname{supp} \varphi \subseteq \Gamma$. Also, let $\varphi_\epsilon:=\epsilon^{-3} \varphi\left(\cdot \epsilon^{-1}\right)$ for $\epsilon>0$ be the mollifier. Clearly, $u * \varphi_\epsilon$, ($\curl \left.u\right) * \varphi_\epsilon$ are smooth in $\mathbb{R}^3$ and converge to $u$ and $\operatorname{curl} u$ respectively in $L^p(\Omega)$ norm.
Next, $\operatorname{curl} \left(u * \varphi_\epsilon\right)=(\operatorname{curl} u ) * \varphi_\epsilon \ + \ w_{\partial} * \varphi_\epsilon $, where $w_{\partial}$ is a distribution supported on $\partial \Omega \cap \operatorname{supp} u$. It follows from \eqref{mollification} that
$$
\operatorname{supp}\left(w_{\partial} * \varphi_\epsilon\right) \subseteq \Gamma+(\partial \Omega \cap \operatorname{supp} u) \subseteq \mathbb{R}^3 \backslash \overline{\Omega} .
$$
Consequently, we have curl $\left(u * \varphi_\epsilon\right) \rightarrow \operatorname{curl} u$ in $L^p(\Omega)$. Similar convergence arguments hold for $v$ and $\curl \,v$. The graph-norm convergence implies convergence of $\gamma_tu_j,\gamma_tv_j$ in $\mathcal X_p(\partial\Omega)$. If $K\Subset\Omega$, trace duality and smoothness of $\Phi_{k_0}(x-\cdot)$ for $x\in K$ give, for every multi-index $\alpha$,
\[
 \sup_{x\in K}|\partial_x^\alpha\mathcal S_{k_0}(f_j-f)(x)|
 \leq C_{K,\alpha}\|f_j-f\|_{\mathcal X_p(\partial\Omega)}.
\]
Consequently the boundary potentials converge locally uniformly with their derivatives. The volume terms converge in $L^p$ by \eqref{Newton-W2p-estimate}. Passing to the limit on compact subsets proves the representation almost everywhere in $\Omega$.
\end{proof}
Next, we prove a result related to the radiation conditions.
\begin{lemma}\label{radiatioN}
If $E^{s}$ and $H^{s}$ are scattering solutions of \eqref{model_EQ} satisfying the radiation condition
\begin{equation}\label{Siver_Muller1}
  \left|E^{s}(x)-\sqrt{\frac{\mu_0}{\varepsilon_0}} H^{s}(x) \times \nu(x)\right| \leq \frac{C}{|x|^2}, \quad \text{for}\ |x| \rightarrow+\infty  
\end{equation}
where $\nu(x)=\frac{x}{|x|}$,
then we have
    \begin{equation}\label{scatt_vec}
       \lim_{R\to\infty} \begin{pmatrix}
   K_{B_R} &  \frac{i}{\omega\varepsilon_{0}} D_{B_R} \\
- \frac{i}{\omega\mu_{0}} D_{B_R} & K_{B_R} 
\end{pmatrix}
\begin{pmatrix}
\nu \times E^{s}|_{\partial B_R}\\ \\
 \nu \times H^{s}|_{\partial B_R}
\end{pmatrix} = 0.  
    \end{equation}
\end{lemma}
\begin{proof}
    The first component of 
    \[
    \begin{pmatrix}
   K_{B_R} &  \frac{i}{\omega\varepsilon_{0}} D_{B_R} \\
- \frac{i}{\omega\mu_{0}} D_{B_R} & K_{B_R} 
\end{pmatrix}
\begin{pmatrix}
\nu \times E^{s}|_{\partial B_R}\\ \\
 \nu \times H^{s}|_{\partial B_R}
\end{pmatrix}
    \]
    can be written as
    \[
    \curl\int_{\partial B_R}\Phi_{k_{0}}(x-y)(\nu \times E^{s}(y)) ds(y) 
    + \frac{i}{\omega\varepsilon_{0}} \curl\, \curl \int_{\partial B_R}\Phi_{k_{0}}(x-y)(\nu \times H^{s}(y)) ds(y). 
    \]
Fix a compact set $K\subset\mathbb R^3$ and write $y=R\nu$ on $\partial B_R$. Uniformly for $x\in K$ and $\nu\in\mathbb S^2$,
\[
 \Phi_{k_0}(x-R\nu)=\frac{e^{ik_0R}}{4\pi R}e^{-ik_0\nu\cdot x}+O_K(R^{-2}),
\]
\[
 \nabla_x\Phi_{k_0}=-ik_0\nu\Phi_{k_0}+O_K(R^{-2}),\qquad
 \nabla_x^2\Phi_{k_0}=-k_0^2\nu\otimes\nu\,\Phi_{k_0}+O_K(R^{-2}).
\]
The outgoing Maxwell far-field expansion gives $E^s(R\nu),H^s(R\nu)=O(R^{-1})$; see \cite{MR0700400,MR1822275}. Put $Z_0=\sqrt{\mu_0/\varepsilon_0}$ and $P_\nu=I-\nu\otimes\nu$. The leading integrand in the first component is
\[
 ik_0\Phi_{k_0}(x-R\nu)
 \big(P_\nu E^s(R\nu)+Z_0\nu\times H^s(R\nu)\big).
\]
The expression in parentheses is $O(R^{-2})$ by \eqref{Siver_Muller1}. The remaining integrand is also $O_K(R^{-3})$, since it is a kernel remainder of order $R^{-2}$ times a field of order $R^{-1}$. Integration over $\partial B_R$ therefore gives an $O_K(R^{-1})$ bound. For the second component the leading integrand is
\[
 ik_0\Phi_{k_0}(x-R\nu)
 \big(P_\nu H^s(R\nu)-Z_0^{-1}\nu\times E^s(R\nu)\big).
\]
Its bracket is $O(R^{-2})$, obtained by taking $\nu\times$ in \eqref{Siver_Muller1}. The same bound follows. Thus \eqref{scatt_vec} holds locally uniformly in $x$.
\end{proof}
\begin{proof}[Proof of Proposition \ref{Lipp_s1}]
Let $R>0$ be such that $\overline{\Omega} \subset B_R=\{x\in\mathbb{R}^3$ such that $|x|<R\}$.
The transmission conditions make the assembled total fields elements of $W^{1,p}(\curl,B_R)$, with no boundary-supported curl terms at $\partial\Omega$. We apply Lemma~\ref{representation} in $B_R$, replacing $u, v$ by $E-E^{i}, H-H^{i}$ respectively, we obtain
\begin{equation}
\begin{aligned}
\begin{pmatrix}
E-E^{i}\\ \\
 H-H^{i}
\end{pmatrix}(x)
&=\int_{B_R} G(x-y)\left(\mathcal{L}_0-\omega\right)\begin{pmatrix}
(E-E^{i})(y)\\ \\
 (H-H^{i})(y)
\end{pmatrix} d y \\
&\qquad -\begin{pmatrix}
   K_{B_R} &  \frac{i}{\omega\varepsilon_{0}} D_{B_R} \\
- \frac{i}{\omega\mu_{0}} D_{B_R} & K_{B_R} 
\end{pmatrix}
\begin{pmatrix}
\nu \times (E-E^{i})|_{\partial B_R}\\ \\
 \nu \times (H-H^{i})|_{\partial B_R}
\end{pmatrix}.
\end{aligned}
\end{equation}
Here the scattered fields $E^{s}$ and $H^{s}$ are defined as $E^{s}=E-E^{i}$ and $H^{s}=H-H^{i}$ and the incident fields $E^{i}, H^{i}$ satisfy the Maxwell system in $\mathbb{R}^3$. Since  $\epsilon =\epsilon_{0}$, $\mu = \mu_{0}$ and $c=c_{\Omega}$ in $\mathbb{R}^3\setminus\overline{\Omega}$, we can simply take $\mathcal{M}=0$ outside $\Omega$. Therefore, by Lemma \ref{radiatioN}, for $x\in\Omega$ we have
\begin{equation}
\begin{pmatrix}
E-E^{i}\\ \\
 H-H^{i}
\end{pmatrix}(x)
=\omega\int_{\Omega} G(x-y)\mathcal{M}\begin{pmatrix}
E(y)\\ \\
 H(y)
\end{pmatrix} d y. 
\end{equation}
Hence the result follows.
\end{proof}

\begin{proof}[Proof of Proposition \ref{well_posed}]
We have shown that every solution of the original problem \eqref{model_EQ} satisfies the Lippmann--Schwinger equation \eqref{Lipp_spacetime}. Conversely, any solution of \eqref{Lipp_spacetime} in $L^p(\Omega)^6$ extends to the whole space $\mathbb{R}^3$ by setting, for $x\in \mathbb{R}^3$,
\begin{equation}
\begin{pmatrix}
E\\ \\
 H
\end{pmatrix}(x)
= \begin{pmatrix}
E^{i}\\ \\
 H^{i}
\end{pmatrix}(x) + \omega\int_{\Omega} G(x-y)\mathcal{M}\begin{pmatrix}
E(y)\\ \\
 H(y)
\end{pmatrix} d y. 
\end{equation}
Extend $\mathcal M$ by zero outside $\Omega$. Since $G$ is the outgoing fundamental matrix of $\mathcal{L}_{0}-\omega$, applying $\mathcal{L}_{0}-\omega$ to the volume representation gives
\begin{equation}\label{eq111}
  (\mathcal{L}_{0}-\omega)  \begin{pmatrix}
E\\ \\
 H
\end{pmatrix}
=
(\mathcal{L}_{0}-\omega)\begin{pmatrix}
E^{i}\\ \\
 H^{i}
\end{pmatrix} + \omega \mathcal{M}\begin{pmatrix}E\\ \\
 H
\end{pmatrix}.
\end{equation}
Since the Lippmann--Schwinger solution belongs to $L^p(\Omega)^6$, $p>1$, equation \eqref{eq111} and the mapping properties of the volume potential imply that $E$ and $H$ belong to $W^{1,p}_{\mathrm{loc}}(\curl,\mathbb{R}^3)$.
Hence
\begin{equation}\label{eq2}
    [E \wedge \nu]|_{\partial \Omega} =  [H \wedge \nu]|_{\partial \Omega} = 0.
\end{equation}
Moreover, for every constant vector $\xi$, the field $G(x-y)\xi$ satisfies the radiation condition with respect to $x$, uniformly for $y\in\Omega$. Consequently, 
\begin{equation}\label{eq3}
     \begin{pmatrix}
E\\ \\
 H
\end{pmatrix}
-
\begin{pmatrix}
E^{i}\\ \\
 H^{i}
\end{pmatrix} \ \text{also satisfies the same radiation condition.}
\end{equation}
Then \eqref{eq111}, \eqref{eq2} and \eqref{eq3} imply that 
$\begin{pmatrix}
E\\ \\
 H
\end{pmatrix}$
is a solution of the original problem \eqref{model_EQ}. 
\end{proof}
\begin{proof}[Proof of Theorem \ref{Main_thm_direct}]
    The proof follows from Propositions \ref{Lipp_s1} and \ref{well_posed}, together with the resolvent estimate \eqref{general-solution-weighted-bound}. In particular, we have
    \[
    \|E\|_{L^p(\Omega)} + \|H\|_{L^p(\Omega)} \leq C \left[ \|E^i\|_{L^p(\Omega)} + \|H^i\|_{L^p(\Omega)}\right].
    \]
    For the incident fields, Lemma~\ref{representation} has zero volume term. For tangential $f\in L^{p,\operatorname{Div}}_{\tan}(\partial\Omega)$, surface integration by parts gives
\[
 \div\mathcal S_{k_0}f=\mathcal S_{k_0}(\operatorname{Div}f),\qquad
 D_\Omega f=k_0^2\mathcal S_{k_0}f+
 \nabla\mathcal S_{k_0}(\operatorname{Div}f).
\]
The classical $L^p$ single-layer and gradient estimates on bounded $C^1$ domains therefore bound $K_\Omega f$ and $D_\Omega f$ in $L^p(\Omega)$ by $C\|f\|_{L^{p,\operatorname{Div}}_{\tan}(\partial\Omega)}$; see \cite{MR501367,MR1438894,MR1317629}. Hence
    \begin{align*}
     \|E\|_{L^p(\Omega)} + \|H\|_{L^p(\Omega)}
     &\leq C \left[ \|E^i\|_{L^p(\Omega)} + \|H^i\|_{L^p(\Omega)}\right] \\
    & \leq 
     C\left[\|\nu \times E^{i}\|_{L_{\tan }^{p, \operatorname{Div}}(\partial \Omega)}+\|\nu \times H^{i}\|_{L_{\tan }^{p, \operatorname{Div}}(\partial \Omega)}\right].        
    \end{align*}   
For every $R>0$ such that $\overline\Omega\subset B_R$, the volume representation and the local mapping properties of the free Maxwell potential give
\begin{align*}
 \|E^{s}\|_{L^{p}(B_R\setminus\overline\Omega)}
 +\|H^{s}\|_{L^{p}(B_R\setminus\overline\Omega)}
 &\leq C_R\big(\|E\|_{L^p(\Omega)}+\|H\|_{L^p(\Omega)}\big)\\
 &\leq C_R\left[\|\nu\times E^{i}\|_{L_{\tan}^{p,\operatorname{Div}}(\partial\Omega)}
 +\|\nu\times H^{i}\|_{L_{\tan}^{p,\operatorname{Div}}(\partial\Omega)}\right].
\end{align*}
This is precisely the asserted $L^p_{\rm loc}$ estimate in the exterior domain.

\end{proof}
\begin{remark}
Note that Theorem \ref{Main_thm_direct} could still be valid for 
 a Lipschitz domain under some restrictions on the range of $p$.  
\end{remark}

\section{Inverse problem }\label{sec7}
In this section, we discuss the inverse scattering problem for the stationary Minkowski-type inclusion described by \eqref{model_EQ}. 
\begin{proof}[Proof of Theorem \ref{inverse_sca}]
We use the standard contradiction argument based on singular incident fields; see \cite{MR2193218}. Assume that the electric far-field patterns of $\Omega_1$ and $\Omega_2$ coincide for all incident and observation directions. The electric incident plane wave, up to the irrelevant multiplicative constant in \eqref{E_inci}, is
\begin{equation}\label{E_inci-inverse-problem-section}
 E^{i}(x;d,\boldsymbol p)=\curl\curl\big(\boldsymbol p e^{ik_0x\cdot d}\big),
 \qquad x\in\mathbb R^3.
\end{equation}
Let $K=B_R$ be a ball with $\overline\Omega_1\cup\overline\Omega_2\Subset K$, and let $D_\infty$ be the unbounded connected component of
$\mathbb R^3\setminus(\overline\Omega_1\cup\overline\Omega_2)$. For $z\notin\overline K$, the regular spherical-wave expansion of $\Phi_{k_0}(\cdot-z)$ in $K$, together with the plane-wave representation of each regular spherical wave, gives Herglotz densities $g_{z,n}\in L^2(\mathbb S^2)$ such that
\[
 {\bf H}_K(g_{z,n})(x):=\int_{\mathbb S^2}e^{ik_0x\cdot d}g_{z,n}(d)\,ds(d)
 \longrightarrow\Phi_{k_0}(x-z)
\]
in $C^3(K')$ for every compact $K'\Subset K$; see \cite{MR0700400}. Applying $\curl\curl(\boldsymbol p\,\cdot)$ gives convergence of the electric incident fields, and applying $(i\omega\mu_0)^{-1}\curl$ also gives convergence of the magnetic fields near both scatterers. Continuity of the scattering map therefore yields equality of the corresponding point-source far fields for $z\notin\overline K$. Rellich's lemma and unique continuation in the observation variable give equality of the scattered fields for $x\in D_\infty$ and $z\notin\overline K$.

To extend this equality in the source variable, fix $x\in D_\infty$. On each compact subset of $D_\infty$, the map
$z\mapsto U^i(\cdot,z)$ is real analytic with values in $L^p(\Omega_j)^6$. The bounded, source-independent inverse $(I-G_{\mathcal M_j})^{-1}$ and the exterior volume representation imply that
$z\mapsto(E_j^s(x,z),H_j^s(x,z))$ is real analytic on $D_\infty$. The identity theorem on this connected set consequently gives
\begin{equation}\label{tota_sam}
 E_1^s(x,z)=E_2^s(x,z),\qquad H_1^s(x,z)=H_2^s(x,z),
 \qquad x,z\in D_\infty.
\end{equation}
This argument uses no reciprocity identity. The scattered fields are regular also at $x=z$ when $z$ is exterior to both scatterers.
We use throughout the point-source fields
\begin{equation}\label{Point-source}
 E^i(x,z):=\curl\curl\big(\boldsymbol p\,\Phi_{k_0}(x-z)\big),
 \qquad
 H^i(x,z):=\frac{1}{i\omega\mu_0}\curl E^i(x,z),
\end{equation}
and write $E_j^t=E^i+E_j^s$ and $H_j^t=H^i+H_j^s$.

For $j=1,2$, let $\mathcal M_j$ denote the effective contrast matrix of the $j$th configuration. Equation \eqref{Lipp_spacetime} gives
\begin{equation}\label{1Lipp_spacetime}
 \binom{E_j^t(\cdot,z)-E^i}{H_j^t(\cdot,z)-H^i}(x)
 =\omega\int_{\Omega_j}G(x-y)\mathcal M_j
 \binom{E_j^t(y,z)}{H_j^t(y,z)}\,dy.
\end{equation}
Suppose that $\Omega_1\ne\Omega_2$. After interchanging the labels if necessary, choose a point of $\Omega_1\setminus\overline\Omega_2$ and connect it to infinity in the connected exterior of $\Omega_2$. A last exit from $\overline\Omega_1$ along such a path gives
\[
 z_0\in\partial\Omega_1\cap\partial D_\infty,
 \qquad z_0\notin\overline\Omega_2.
\]
With $\nu_1(z_0)$ the outward normal, take
$z_n=z_0+t_n\nu_1(z_0)\in D_\infty$, where $t_n\downarrow0$. The $C^1$ boundary gives
$d_n:=\dist(z_n,\Omega_1)=t_n(1+o(1))$.
Since the point source remains a positive distance from $\Omega_2$, the direct estimate applied to \eqref{1Lipp_spacetime} shows that
\[
 \|E_2^t(\cdot,z_n)\|_{L^p(\Omega_2)}
 +\|H_2^t(\cdot,z_n)\|_{L^p(\Omega_2)}
\]
is uniformly bounded. Consequently,
\[
\begin{pmatrix}
 E_{2}^{s}(x,z_n)\\ \\
 H_{2}^{s}(x,z_n)
\end{pmatrix},
\]
is uniformly bounded for $x$ in a fixed neighborhood of $z_0$ whose closure is disjoint from $\overline\Omega_2$, and hence for $x=z_n$ when $n$ is sufficiently large.
On the other hand Lemma \ref{blowup}, see below, implies that
\[
\begin{pmatrix}
 E^s_{1}(z_n,z_n)\\ \\
 H^s_{1}(z_n,z_n)
\end{pmatrix}
\]
is unbounded, which contradicts \eqref{tota_sam}. Hence $\Omega_1$ must be identical with $\Omega_2.$
\end{proof}

\begin{lemma}\label{blowup}
Let $(z_n)$ be the sequence used above and let $d_n:=d(z_n,\Omega_1)$. There exist constants $c_{\rm st}>0$ and $C_{\rm corr}>0$, independent of $n$ and of $V_1$ in a fixed subcritical neighborhood, such that
\begin{equation}\label{quantitative-blowup}
 |E_1^s(z_n,z_n)|+|H_1^s(z_n,z_n)|
 \geq
 \left(c_{\rm st}-C_{\rm corr}\frac{\eta_{V_1}}{1-\eta_{V_1}}\right)d_n^{-3}
 +o(d_n^{-3}).
\end{equation}
Consequently, if $\eta_{V_1}\leq\eta_{\rm inv}$ and $\eta_{\rm inv}>0$ is chosen so that
\begin{equation}\label{eta-inv-choice}
 C_{\rm corr}\frac{\eta_{\rm inv}}{1-\eta_{\rm inv}}
 \leq\frac{c_{\rm st}}{2},
\end{equation}
then the sequence $(E_1^s(z_n,z_n),H_1^s(z_n,z_n))$ is unbounded.
\end{lemma}

\begin{proof}
We suppress the subscript $1$ on the material parameters and operators. Set
\[
 U^t:=\binom{E_1^t}{H_1^t},\qquad
 U^i:=\binom{E^i}{H^i},\qquad
 U^s:=\binom{E_1^s}{H_1^s}.
\]
With the decomposition \eqref{GM-decomposition}, the Lippmann--Schwinger equation reads
\begin{equation}\label{inverse-LS-split}
 U^s=G_{\rm st}U^t+G_{\rm mov}U^t=:I_1+I_2.
\end{equation}
Here, we recall that $G_{\rm st}$ is the volume operator of the associated static configuration $V=0$ and $G_{\rm mov}=\sum_{j\geq0}G_j$ contains all $V$-dependent terms.

Let
\[
 \widetilde U^t:=\binom{\widetilde E_1^t}{\widetilde H_1^t}
\]
be the total field for the associated static transmission problem, with the same point-source incident field. Thus, in $\Omega_1$,
\[
 \curl\widetilde E_1^t=i\omega\mu\widetilde H_1^t,
 \qquad
 \curl\widetilde H_1^t=-i\omega\varepsilon\widetilde E_1^t,
\]
and
\begin{equation}\label{static-LS-inverse-section}
 K_{\rm st}\widetilde U^t=U^i,
 \qquad
 \widetilde U^s:=\widetilde U^t-U^i=G_{\rm st}\widetilde U^t.
\end{equation}
Consequently,
\begin{equation}\label{I1-static-plus-error}
 I_1=\widetilde U^s+A,
 \qquad
 A:=G_{\rm st}(U^t-\widetilde U^t).
\end{equation}
Subtracting the static equation from the full $V$-dependent equation gives
\begin{equation}\label{moving-static-difference}
 K_{\rm st}(U^t-\widetilde U^t)=G_{\rm mov}U^t.
\end{equation}
Let $\rho_{V_1}:=\|K_{\rm st}^{-1}G_{\rm mov}\|_*$. By \eqref{Gmov-sharp-bound}, $\rho_{V_1}\leq\eta_{V_1}$. The resolvent estimate \eqref{solution-weighted-bound} and \eqref{moving-static-difference} therefore imply
\begin{align}
 \|U^t\|_{p,*}
 &\leq\frac{\kappa_{\rm st}}{1-\rho_{V_1}}\|U^i\|_{p,*},
 \label{inverse-total-field-bound}\\
 \|U^t-\widetilde U^t\|_{p,*}
 &\leq\frac{\rho_{V_1}\kappa_{\rm st}}{1-\rho_{V_1}}\|U^i\|_{p,*}
 \leq\frac{\eta_{V_1}\kappa_{\rm st}}{1-\eta_{V_1}}\|U^i\|_{p,*}.
 \label{inverse-moving-static-bound}
\end{align}
It follows that
\begin{align}
 \|A\|_{p,*}
 &\leq
 \|G_{\rm st}\|_*\frac{\eta_{V_1}\kappa_{\rm st}}{1-\eta_{V_1}}\|U^i\|_{p,*},
 \label{A-Lp-bound}\\
 \|I_2\|_{p,*}
 &\leq
 \frac{\eta_{V_1}}{1-\eta_{V_1}}\|U^i\|_{p,*}.
 \label{I2-Lp-bound}
\end{align}
These estimates make explicit that both corrections vanish with the $V$-dependent parameter, whereas the full static contrast is kept in $K_{\rm st}$.

For completeness, we record the pointwise estimate needed in the singular-source argument. From the explicit form \eqref{fund_soll} of the free Maxwell tensor, for every fixed $R>0$,
\begin{equation}\label{Maxwell-kernel-pointwise-bound}
 |G(x)|\leq C_R\big(|x|^{-3}+|x|^{-2}+|x|^{-1}\big),
 \qquad 0<|x|\leq R.
\end{equation}
Let $q=p/(p-1)$. Since $p>3$, one has $1<q<3/2$. To make the distance dependence explicit, split the integral into a fixed ball centered at $z_n$ and its complement. For the leading singularity,
\[
 \int_{\Omega_1\cap B_R(z_n)}|z_n-y|^{-3q}\,dy
 \leq C\int_{d_n}^{R}r^{2-3q}\,dr
 \leq C d_n^{3-3q}.
\]
The $|x|^{-2}$ and $|x|^{-1}$ terms in \eqref{Maxwell-kernel-pointwise-bound} are less singular (and the contribution outside $B_R(z_n)$ is uniformly bounded). Therefore, using $3/q-3=-3/p$,
\begin{equation}\label{kernel-Lq-distance-bound}
 \|G(z_n-\cdot)\|_{L^q(\Omega_1)}\leq C d_n^{-3/p}.
\end{equation}
Moreover, away from the pole,
\[
 E^i(x,z_n)=\nabla\big(\boldsymbol p\cdot\nabla\Phi_{k_0}(x-z_n)\big)
              +k_0^2\boldsymbol p\,\Phi_{k_0}(x-z_n),
\]
so $|E^i(x,z_n)|\leq C(|x-z_n|^{-3}+|x-z_n|^{-1})$, while
$H^i=(i\omega\mu_0)^{-1}\curl E^i$ has order $|x-z_n|^{-2}$. Hence
\[
 \|E^i(\cdot,z_n)\|_{L^p(\Omega_1)}^p
 \leq C\int_{d_n}^{R}r^{2-3p}\,dr+C
 \leq C d_n^{3-3p},
\]
and the magnetic contribution is of the smaller order $d_n^{3/p-2}$. Thus, for $d_n$ small,
\begin{equation}\label{point-source-volume-bound}
 \|U^i(\cdot,z_n)\|_{p,*}\leq C d_n^{3/p-3}.
\end{equation}
The densities defining $A=\mathcal V_\omega\big(\mathcal M_{\rm st}(U^t-\widetilde U^t)\big)$ and $I_2=\mathcal V_\omega(\mathcal M_{\rm mov}U^t)$ are bounded, by \eqref{inverse-total-field-bound}--\eqref{inverse-moving-static-bound}, \eqref{effective-contrast-expanded}, and the definition of $\eta_{V_1}$, by
\[
 C\frac{\eta_{V_1}}{1-\eta_{V_1}}\|U^i(\cdot,z_n)\|_{p,*}.
\]
Combining this estimate with \eqref{kernel-Lq-distance-bound}--\eqref{point-source-volume-bound}, and using
$-3/p+(3/p-3)=-3$, yields
\begin{equation}\label{moving-correction-pointwise}
 |A(z_n,z_n)|+|I_2(z_n,z_n)|
 \leq C_{\rm corr}\frac{\eta_{V_1}}{1-\eta_{V_1}}d_n^{-3}.
\end{equation}
The constant $C_{\rm corr}$ absorbs the fixed static factors and may be chosen uniformly for $V_1$ in a fixed subcritical neighborhood.

It remains to justify the principal static singularity, i.e. of $I_1$, at the $C^1$ interface. We give the local rescaling argument, so that no curved-interface Green-tensor asymptotic is assumed. Translate and rotate coordinates so that $z_0=0$, $\nu_1(z_0)=e_3$, and $\Omega_1$ lies locally below a $C^1$ graph with horizontal tangent plane at the origin. Write $z_t=te_3$, $D_t=t^{-1}\Omega_1$, and $\mathbb H_-:=\{\xi_3<0\}$. Then
\[
 \chi_{D_t}\longrightarrow\chi_{\mathbb H_-}\quad\hbox{locally almost everywhere},
 \qquad \dist(z_t,\Omega_1)=t(1+o(1)).
\]
Extend the rescaled interior static fields by zero:
\[
 f_t(\xi):=t^3\chi_{D_t}(\xi)\widetilde E_1^t(t\xi,z_t),\qquad
 h_t(\xi):=t^3\chi_{D_t}(\xi)\widetilde H_1^t(t\xi,z_t).
\]
A direct change of variables $y=t\xi$ gives
\[
 \|f_t\|_{L^p(\mathbb R^3)}
 =t^{3-3/p}\|\widetilde E_1^t(\cdot,z_t)\|_{L^p(\Omega_1)},
 \qquad
 \|h_t\|_{L^p(\mathbb R^3)}
 =t^{3-3/p}\|\widetilde H_1^t(\cdot,z_t)\|_{L^p(\Omega_1)}.
\]
Since $K_{\rm st}\widetilde U^t=U^i$, the static resolvent estimate and \eqref{point-source-volume-bound} give
\[
 \|\widetilde U^t\|_{p,*}\leq \kappa_{\rm st}\|U^i(\cdot,z_t)\|_{p,*}
 \leq C d_t^{3/p-3},
 \qquad d_t:=\dist(z_t,\Omega_1)=t(1+o(1)).
\]
The weights in $\|\cdot\|_{p,*}$ are fixed, so each component is controlled by this norm. Consequently,
\[
 \|f_t\|_{L^p(\mathbb R^3)}+\|h_t\|_{L^p(\mathbb R^3)}
 \leq C t^{3-3/p}d_t^{3/p-3}
 =C\left(\frac{t}{d_t}\right)^{3-3/p}\leq C,
\]
for all sufficiently small $t$. This is the uniform bound needed for weak compactness.
Along a subsequence, $f_t\rightharpoonup f$ and $h_t\rightharpoonup h$ weakly in $L^p(\mathbb R^3)^3$, where $f,h$ vanish outside $\overline{\mathbb H_-}$.

Let $\mathcal N=\nabla\nabla\!\cdot\,(-\Delta)^{-1}$ be the bounded whole-space singular integral with Fourier multiplier $-\xi\otimes\xi/|\xi|^2$. Thus $\mathcal N$ uses the distributional Hessian of $\Phi_0$ and includes its local term. Put $\alpha=r_\varepsilon-1$ and $\gamma=r_\mu-1$. Rescaling the exact static volume formula \eqref{volume-potential-components} gives, on every fixed bounded set,
\begin{align*}
 t^3\widetilde E_1^s(t\xi,z_t)&=\alpha\mathcal N f_t(\xi)+R_t^E(\xi),\\
 t^3\widetilde H_1^s(t\xi,z_t)&=\gamma\mathcal N h_t(\xi)+R_t^H(\xi),
\end{align*}
where $R_t^E,R_t^H\to0$ locally uniformly. Indeed, the rescaled Helmholtz wave number is $tk_0$. The Hessian difference $\nabla^2(\Phi_{tk_0}-\Phi_0)$ and the zeroth-order term $(tk_0)^2\Phi_{tk_0}$ are bounded by $Ct^2/|\xi-\eta|$ on the integration range $|\xi-\eta|\leq C/t$. The cross terms are bounded by
$C(t|\xi-\eta|^{-2}+t^2|\xi-\eta|^{-1})$. Since $q=p/(p-1)<3/2$,
\[
 \big\|t|\cdot|^{-2}+t^2|\cdot|^{-1}\big\|_{L^q(B_{C/t})}
 \leq C t^{3/p}.
\]
H\"older's inequality and the uniform bounds on $f_t,h_t$ therefore give $R_t^E,R_t^H=O(t^{3/p})$ on each fixed bounded set.

On compact subsets of $\mathbb H_-$, the scaled incident electric field tends to
$e^i(\xi):=\nabla^2\Phi_0(\xi-e_3)\boldsymbol p$, and the scaled incident magnetic field tends to zero. Passing weakly in the interior volume equations yields
\begin{equation}\label{flat-limit-equations}
 f=\chi_{\mathbb H_-}(e^i+\alpha\mathcal N f),\qquad
 h=\gamma\chi_{\mathbb H_-}\mathcal N h.
\end{equation}
These flat equations have at most one $L^p$ solution for each right-hand side when the relative coefficient is positive. To see this, let
$w=(r-1)\chi_{\mathbb H_-}\mathcal Nw$, with $r>0$, and set $F=\mathcal Nw$. The identities
\[
 \curl F=0,\qquad \div F=-\div w
\]
show that $F=\nabla u$ and
$\div((1+(r-1)\chi_{\mathbb H_-})\nabla u)=0$ distributionally. Here $u\in W^{1,p}_{\rm loc}(\mathbb R^3)$, its restrictions $u_\pm$ are harmonic, and their traces satisfy $u_+=u_-$ and $\partial_3u_+=r\partial_3u_-$ on the plane. Define
\[
 v(\xi',s)=
 \begin{cases}
 u_+(\xi',s),&s>0,\\[1mm]
 \dfrac{2r}{1+r}u_-(\xi',s)+\dfrac{1-r}{1+r}u_+(\xi',-s),&s<0.
 \end{cases}
\]
Both the value and the normal derivative of $v$ match across the plane, so $v$ is entire harmonic. Moreover, $\nabla v\in L^p(\mathbb R^3)$. The mean-value estimate for its harmonic derivatives implies $\nabla v=0$. The definition of $v$ then gives $F=0$, hence $w=0$. This proves the asserted uniqueness and, in particular, $h=0$.

The electric flat solution is obtained by the image construction. Set
\[
 q_\varepsilon=\frac{1-r_\varepsilon}{1+r_\varepsilon},\qquad
 \tau_\varepsilon=\frac{2}{1+r_\varepsilon},\qquad
 R=I-2e_3\otimes e_3.
\]
The interior total field and the exterior reflected field are respectively
\[
 f(\xi)=\tau_\varepsilon\chi_{\mathbb H_-}(\xi)
             \nabla^2\Phi_0(\xi-e_3)\boldsymbol p,
 \qquad
 e^{\rm ref}(\xi)=q_\varepsilon\nabla^2\Phi_0(\xi+e_3)R\boldsymbol p
 \quad(\xi_3>0).
\]
For completeness, their scalar potentials are $\tau_\varepsilon\boldsymbol p\cdot\nabla\Phi_0(\xi-e_3)$ below the plane and
$\boldsymbol p\cdot\nabla\Phi_0(\xi-e_3)+q_\varepsilon(R\boldsymbol p)\cdot\nabla\Phi_0(\xi+e_3)$ above it. The identities $1+q_\varepsilon=\tau_\varepsilon$ and $1-q_\varepsilon=r_\varepsilon\tau_\varepsilon$ verify continuity of the potential and the weighted normal derivative. Thus the associated scattered field is globally curl-free, has divergence $-\alpha\div f$, and belongs to $L^p(\mathbb R^3)^3$. Its difference from $\alpha\mathcal Nf$ is an entire harmonic $L^p$ field and hence vanishes. This verifies the limiting electric equation and gives $e^{\rm ref}=\alpha\mathcal Nf$ in the upper half-space. The construction is the zero-frequency principal part of the flat two-medium transmission problem; compare \cite{MR1608084}.

It remains to pass to the diagonal exterior evaluation. For small $t$, the supports of $f_t,h_t$ stay a fixed positive distance from $e_3$. Away from that ball, the kernel $\nabla^2\Phi_0(e_3-\cdot)$ belongs to $L^q(\mathbb R^3)$, because it decays as $|\xi|^{-3}$. Consequently weak $L^p$ convergence implies
$\mathcal N f_t(e_3)\to\mathcal N f(e_3)$ and
$\mathcal N h_t(e_3)\to0$. The uniqueness just proved makes these limits independent of the subsequence. Since
$\nabla^2\Phi_0(2e_3)=(3e_3\otimes e_3-I)/(32\pi)$, returning to the original coordinates gives
\begin{equation}\label{static-Green-comparison}
 \begin{aligned}
 d_n^3\widetilde E_1^s(z_n,z_n)&\longrightarrow
 \frac{\varepsilon-\varepsilon_0}{\varepsilon+\varepsilon_0}
 \frac{(I+\nu_1(z_0)\otimes\nu_1(z_0))\boldsymbol p}{32\pi},\\
 d_n^3\widetilde H_1^s(z_n,z_n)&\longrightarrow0.
 \end{aligned}
\end{equation}
The nondegeneracy assumption is therefore precisely $\varepsilon\neq\varepsilon_0$ for the present electric-dipole normalization and a nonzero $\boldsymbol p$. Since
$|(I+\nu\otimes\nu)\boldsymbol p|\geq|\boldsymbol p|$, one may take
\[
 c_{\rm st}=\frac{|\varepsilon-\varepsilon_0|}{\varepsilon+\varepsilon_0}
             \frac{|\boldsymbol p|}{64\pi}>0.
\]
In particular,
\begin{equation}\label{static-singular-lower-bound}
 |\widetilde E_1^s(z_n,z_n)|+|\widetilde H_1^s(z_n,z_n)|
 \geq c_{\rm st}d_n^{-3}+o(d_n^{-3}).
\end{equation}
Combining \eqref{inverse-LS-split}, \eqref{I1-static-plus-error}, \eqref{moving-correction-pointwise}, and \eqref{static-singular-lower-bound} proves \eqref{quantitative-blowup}. Choosing $\eta_{\rm inv}$ according to \eqref{eta-inv-choice} completes the proof. For example, for the two configurations one may choose
\[
 \eta_{\rm inv}=\min\left\{\frac12,
       \min_{j=1,2}\frac{c_{{\rm st},j}}{2C_{{\rm corr},j}+c_{{\rm st},j}}\right\}.
\]
\end{proof}
\appendix
\section{A space--time modulation leading to the Minkowski-type law}\label{app:spacetime}
This appendix records a formal derivation of the constitutive law used in the paper.
Its purpose is only to identify a physical route to the model; none of the arguments in the
previous sections depends on this derivation. We distinguish the velocity of a traveling
modulation, denoted by $U$, from the constitutive vector $V$ used in the main text.

\subsection{A traveling profile in a co-moving frame}
Let $U=u\widehat U$, $0<u<c$, and consider an isotropic space--time medium
\begin{equation}\label{app:traveling-profile}
 D(x,t)=\varepsilon(x-Ut)E(x,t),\qquad
 B(x,t)=\mu(x-Ut)H(x,t).
\end{equation}
The functions $\varepsilon$ and $\mu$ may be constant, periodic, or locally varying.
Introduce the Lorentz variables
\[
 y_\perp=x_\perp,\qquad y_\parallel=\gamma_u(x_\parallel-ut),\qquad
 \tau=\gamma_u\left(t-\frac{ux_\parallel}{c^2}\right),\qquad
 \gamma_u=(1-u^2/c^2)^{-1/2},
\]
where $x_\parallel=x\cdot\widehat U$, and the parallel and perpendicular components
are taken relative to $\widehat U$. Since
\[
 x-Ut=y_\perp+\gamma_u^{-1}y_\parallel\widehat U,
\]
the constitutive profiles become independent of $\tau$ in the co-moving variables.

The Lorentz transformation of the electromagnetic fields reads
\begin{align*}
 E'_\parallel&=E_\parallel,&
 E'_\perp&=\gamma_u(E+U\times B)_\perp,\\
 B'_\parallel&=B_\parallel,&
 B'_\perp&=\gamma_u\left(B-\frac{1}{c^2}U\times E\right)_\perp,\\
 D'_\parallel&=D_\parallel,&
 D'_\perp&=\gamma_u\left(D+\frac{1}{c^2}U\times H\right)_\perp,\\
 H'_\parallel&=H_\parallel,&
 H'_\perp&=\gamma_u(H-U\times D)_\perp.
\end{align*}
Using the inverse transformation in \eqref{app:traveling-profile} gives, pointwise in $y$,
\begin{equation}\label{app:transformed-relations}
\begin{aligned}
 D'-\frac{1}{c^2}U\times H'&=\varepsilon_U(y)(E'-U\times B'),\\
 B'+\frac{1}{c^2}U\times E'&=\mu_U(y)(H'+U\times D'),
\end{aligned}
\end{equation}
where
\[
 \varepsilon_U(y)=\varepsilon(y_\perp+\gamma_u^{-1}y_\parallel\widehat U),\qquad
 \mu_U(y)=\mu(y_\perp+\gamma_u^{-1}y_\parallel\widehat U).
\]
Setting $V=-U$ and dropping the primes, \eqref{app:transformed-relations} becomes
\begin{equation}\label{app:minkowski-law}
\left\{\begin{aligned}
 D+\frac{1}{c^2}V\times H&=\varepsilon_U(y)(E+V\times B),\\
 B-\frac{1}{c^2}V\times E&=\mu_U(y)(H-V\times D).
\end{aligned}\right.
\end{equation}
Thus the implicit law \eqref{relation*} is obtained exactly from a traveling isotropic
profile after passing to the co-moving description. This calculation is also the basic
mechanism behind the moving-type bianisotropy appearing in homogenized space--time
media; see \cite{Huidobro2019,Huidobro2021}. The sign of $V$ depends only on whether
one records the velocity of the modulation or the velocity of the transformed frame.

For constant $\varepsilon$ and $\mu$, let $K=V\times$, $P_\parallel=\widehat V\otimes\widehat V$,
and $P_\perp=I-P_\parallel$. Solving \eqref{app:minkowski-law} for $D$ and $B$ gives
\begin{align*}
 D&=\varepsilon(P_\parallel+a_VP_\perp)E+\kappa_V V\times H,\\
 B&=\mu(P_\parallel+a_VP_\perp)H-\kappa_V V\times E,
\end{align*}
where
\[
 a_V=\frac{1-|V|^2/c^2}{1-\varepsilon\mu|V|^2},\qquad
 \kappa_V=\frac{\varepsilon\mu-c^{-2}}{1-\varepsilon\mu|V|^2}.
\]
Hence the factor $1-\varepsilon\mu|V|^2=1-\beta^2$ is already present at the constitutive
level. It is not introduced by the Born expansion.

\subsection{Effective stationary media}
For a fixed macroscopic sample one may instead impose a rapid traveling modulation
inside the stationary region $\Omega$, for example
\[
 \varepsilon^\delta(x,t)=\varepsilon_m\left(\frac{\widehat U\cdot x-ut}{\delta}\right),\qquad
 \mu^\delta(x,t)=\mu_m\left(\frac{\widehat U\cdot x-ut}{\delta}\right),\qquad x\in\Omega,
\]
with $\delta$ small compared with the wavelength. The co-moving transformation makes
the cell stationary, and the long-wave homogenization produces a time-independent
bianisotropic law. For one-dimensional traveling modulations it has the form
\begin{equation}\label{app:effective-law}
\begin{aligned}
 D&=(\varepsilon_\parallel P_\parallel+\varepsilon_\perp P_\perp)E
       +\xi\widehat U\times H,\\
 B&=(\mu_\parallel P_\parallel+\mu_\perp P_\perp)H
       -\xi\widehat U\times E;
\end{aligned}
\end{equation}
see \cite{Huidobro2019,Huidobro2021}. A generic space--time modulation yields a more
general bianisotropic medium. The effective law \eqref{app:effective-law} belongs exactly
to the three-parameter Minkowski class used in the present paper if there are
$\varepsilon,\mu>0$ and $s\geq0$ such that, with $V=s\widehat U$,
\[
 \varepsilon_\parallel=\varepsilon,\qquad \mu_\parallel=\mu,\qquad
 \frac{\varepsilon_\perp}{\varepsilon}=\frac{\mu_\perp}{\mu}
 =\frac{1-s^2/c^2}{1-\varepsilon\mu s^2},\qquad
 \xi=s\frac{\varepsilon\mu-c^{-2}}{1-\varepsilon\mu s^2}.
\]
These are the Minkowski closure relations for the effective coefficients. They explain
why the model considered in the main text should be viewed as a constrained subclass
of general bianisotropic space--time media. Isorefractive Minkowskian space--time crystals
give a particularly direct realization of this correspondence \cite{Prudencio2023}.

\subsection{Locally varying material parameters}
The preceding calculation does not require $\varepsilon$ and $\mu$ to be constant.
If the co-moving profiles are $\varepsilon(y)$ and $\mu(y)$ while $V$ remains constant, define
\[
 c_\Omega(y)=\frac{1}{\sqrt{\varepsilon(y)\mu(y)}},\qquad
 \beta(y)=\frac{|V|}{c_\Omega(y)}=|V|\sqrt{\varepsilon(y)\mu(y)}.
\]
With
\[
 \mathbb S_V(y)=\begin{pmatrix}
 0&-\mu(y)V\times\\
 \varepsilon(y)V\times&0
 \end{pmatrix},
\]
the same pointwise algebra gives
\[
 \mathbb S_V(y)^3=\beta(y)^2\mathbb S_V(y),
\]
and, whenever $\beta(y)\neq1$,
\[
 (I-\mathbb S_V(y))^{-1}
 =I+\frac{\mathbb S_V(y)+\mathbb S_V(y)^2}{1-\beta(y)^2}.
\]
Hence the natural uniform subcritical condition is
\[
 \operatorname*{ess\,sup}_{y\in\Omega}\beta(y)<1.
\]
The possible set $\{y:\beta(y)=1\}$ is a local constitutive critical set. Extending the
scattering and inverse results of this paper to such variable coefficients, in particular
near this set, requires additional PDE estimates and is not pursued here.

\bibliography{math1} 

\bibliographystyle{alpha}
\end{document}